%% file: yang.tex
\documentclass[reqno]{amsart}
\newtheorem{lemma}{Lemma}
\newtheorem{definition}{Definition}
\newtheorem{theorem}{Theorem}

\usepackage{amsmath, amsthm, amsfonts, amssymb}
\usepackage{graphicx}
\usepackage{esint}
\usepackage{float}
\usepackage[a4paper, margin=1in]{geometry}

\usepackage{xcolor}  

\newcommand{\ARXIV}{\texttt{arXiv}}

\title[Convergence Rates in Stochastic Homogenization]{Influence of Lower-Order Terms on the Convergence Rates in Stochastic Homogenization of Elliptic Equations}

\author{Man Yang}
\address{Kyushu University, Japan.}
\email{yang.man.365@s.kyushu-u.ac.jp}

\begin{document}

\begin{abstract}
In this study, we investigate the convergence rates for the homogenization of elliptic equations with lower-order terms, under the spectral gap assumption, in both bounded domains and the entire space. Our analysis demonstrates that lower-order terms significantly affect the convergence rate, particularly in the full space where the rate changes from \(O(\epsilon)\) (under the spectral gap condition without lower-order terms, see \cite[Theorem 2]{Fischer2021}) to \(O(\epsilon^{d/({d+2})})\) due to their influence.  In contrast, in bounded domains, the convergence rate remains \(O(\epsilon^{1/2})\), as the boundary conditions exert a stronger influence than the lower-order terms. To manage the complexities introduced by lower-order terms, we developed a novel technique that localizes the analysis within small grids, enabling the application of the Poincaré inequality for effective estimates. This work advances the framework in \cite{Fischer2021}, offering a refined approach to quantitative homogenization with lower-order terms.
\end{abstract}


\maketitle





\input{2.Introduction}
\input{3.Preliminary}


\input{4.Estimations_for_correctors}


\input{5.RSPACE}

\input{6.bounded_domain}

\input{gragh}

\section*{Acknowledgment}
This work was supported by the WISE program (MEXT) at Kyushu University. I would like to extend my sincerest gratitude to Professor Osada Hirofumi for his invaluable guidance and revisions throughout the development of this paper. I am also deeply grateful to Professor Stefan Neukamm for the fruitful discussions and for providing numerous ideas that helped to address key challenges in this research during my visit to TU Dresden.

\input{References}
                                                    
\end{document}

%% file: 2.Introduction.tex
\section{Introduction }\label{sec 1}

In many applications, systems governed by partial differential equations (PDEs) have heterogeneous coefficients, where small-scale variations impact large-scale behavior. Homogenization theory approximates these systems with simpler "homogenized" equations reflecting their overall behavior. Stochastic homogenization specifically addresses cases with random coefficients. Qualitative homogenization establishes the existence of this limit as \( \epsilon \to 0 \), without specifying the convergence rate. In contrast, quantitative homogenization provides precise estimates on rates and error bounds, essential for accurately understanding the effects of randomness and fine-scale variations on large-scale behavior. 

This work focuses on the quantitative stochastic homogenization of elliptic equations with lower-order terms, aiming to assess their impact on convergence rates. Lower-order terms are common in practical models but add complexity, influencing both convergence behavior and the homogenized structure. Most quantitative studies have focused on models without lower-order terms. However, understanding the quantitative effects of lower-order terms in stochastic settings remains an important open question. We consider elliptic PDEs of the form:
\[
-\nabla\cdot(\mathbf{a}^\epsilon(x)\nabla u^\epsilon(x))+\mathbf b^\epsilon (x) \nabla u^\epsilon(x)+\Lambda u^\epsilon(x) = f(x) \quad \text{in } \mathbb{R}^d,
\]
where \(\mathbf{a}^\epsilon(x)\) and \(\mathbf b^\epsilon(x)\) are random fields with correlations decaying quickly on scales larger than the microscopic scale \(\epsilon\). Here, \(\Lambda\) is a constant that ensures control over the influence of the lower-order term \(\mathbf b^\epsilon (x)\), particularly necessary for elliptic equations. Our goal is to extend the qualitative results from work \cite{wodene}, which established solution convergence, to a quantitative analysis that precisely determines the rates of convergence. The homogenized problem is given by:
\[
-\nabla\cdot(\mathbf{\bar{a}}\nabla  u_{0}(x))+\mathbf{\bar{b}} \nabla  u_{0}(x)+\Lambda u_0(x) = f(x) \quad \text{in } \mathbb{R}^d,
\]
where \(\mathbf{\bar{a}}\) and \(\mathbf{\bar{b}}\) represent homogenized coefficients. 

We prove that for \(d \geq 3\), the convergence rate is \(O\left(\epsilon^{d/(d+2)}\right)\), while in bounded domains, it remains \(O(\epsilon^{1/2})\), with boundary conditions having a greater influence than the lower-order terms. Moreover, our study introduces a novel localization technique within the stochastic framework, enabling refined error estimates in quantitative homogenization. This approach addresses a significant gap in existing homogenization theory, providing a robust method to assess the error impact of lower-order terms in models with random coefficients and broadening the applicability of homogenization to a wider range of PDEs with practical relevance.

In qualitative homogenization, \cite{modica1986nonlinear} and \cite{dal1986nonlinear} used average integration for the convenience of applying the subadditive ergodic theorem \cite{akcoglu1981ergodic} to achieve convergence. For the cases with lower order terms, the work on parabolic equations \cite{wodene} established that the solution \(u_\epsilon\) converges to \(u_0\). While qualitative homogenization results are typically obtained under the conditions of stationarity and ergodicity, ensuring statistical uniformity of the medium's properties over large scales, obtaining quantitative results, particularly convergence rates, requires stronger conditions. Recent advancements in quantitative homogenization have provided precise estimates of convergence rates and error bounds. For instance, works such as \cite{armstrong2018quantitative} and \cite{armstrong2019quantitative} have established quantitative results under the unit range of dependence assumption, while \cite{GloriaAntoine0} developed new quantitative methods based on spectral gap estimates with respect to Glauber dynamics on coefficient fields. 
 
Lower-order terms in practical problems add complexity to quantitative homogenization. In periodic settings, convergence rates for systems with lower-order terms can be directly obtained, as shown in \cite{wang2022homogenization} and \cite{jikov2012homogenization}. By contrast, for stochastic cases, where coefficients are random, further investigation is required to establish the convergence rates. To address this gap, we extend the framework introduced in \cite{Fischer2021} to consider elliptic equations with lower-order terms under stochastic conditions. 

This paper is organized as follows: In Section \ref{sec 1}, we introduce the motivation and objectives of our study, review relevant literature, and present essential notations, setups, and assumptions. This section also states the main theorems. Section \ref{sec 2} provides preliminary results for the analysis. Section \ref{sec 3} presents the main estimates, where we derive quantitative results crucial for proving the main theorems. Section \ref{sec 4} contains the detailed proofs of the main theorems, building upon the estimates from Section \ref{sec 3}. Finally, Section \ref{sec 5} provides examples of cases that satisfy the spectral gap condition, offering insight into understanding scenarios where this condition holds.

\subsection{Notations}
Let \( d \in \mathbb{N} \) denote the spatial dimension. Throughout this work, the letter \(C\) denotes positive constants that may vary from line to line. These constants \(C\) depend only on \(d\), \(\lambda\), \(K\), and \(\Lambda\), which are specified in Assumptions \textbf{A1}–\textbf{A5}. Importantly, \(C\) does not depend on other variables like \(\epsilon\) or quantities that vary with \(\epsilon\). \( A \lesssim B \) means that \( A \) is less than or equal to \( B \) up to a constant multiple, i.e., \( A \leq C B \) for some constant \( C > 0 \). 

For a function \(f: \mathbb{R}^{d} \to \mathbb{R}\), its partial derivative with respect to \(x_i\) is denoted by \(\partial_{i}f\), and its gradient is given by \(\nabla f := (\partial_{1}f, \partial_{2}f, \dots, \partial_{d}f)^T\). If \(\mathbf{f}: \mathbb{R}^{d} \to \mathbb{R}^{d}\) is a vector-valued function with \(\mathbf{f} = (f_{1}, \dots, f_{d})\), the divergence of \(\mathbf{f}\) is given by \(\nabla \cdot \mathbf{f} = \sum_{i=1}^{d}\partial_{i}f_{i}\). Throughout this paper, we adopt the Einstein summation convention, where repeated indices imply summation. For example, the expression \(f_i g_i\) denotes \(\sum_{i=1}^{d} f_i g_i\).
For a measurable set \(U \subset \mathbb{R}^d\) with positive measure, the average integral of a function \(f\) over \(U\) is defined as:
\[
\fint_U f(x) \, \mathrm{d}x = \frac{1}{|U|} \int_U f(x) \, \mathrm{d}x,
\]
where \(|U|\) is the Lebesgue measure of the set \(U\). We consider an open ball \(B_{R}(x_0) := \{x \in \mathbb{R}^{d} \mid |x - x_0| < R\}\), where \(R > 0\) and \(x_0 \in \mathbb{R}^d\). For simplicity, we write \(B_R := B_R(0)\) when the ball is centered at the origin. 
Let \(\mathbb{P}\) denote a probability measure on a probability space \((\Omega, \mathcal{F}, \mathbb{P})\), and let \(\mathbb{E}[X]\) denote the expectation of a measurable random variable \(X\) with respect to \(\mathbb{P}\).

\(\mathbb{R}_\text{skew}^{d \times d}\) denotes the space of \(d \times d\) skew-symmetric matrices, i.e., matrices \(A = (a_{ij})\) satisfying \(a_{ij} = -a_{ji}\) for all \(1 \leq i, j \leq d\). For \(k \in [1, \infty]\), let \(W^{1,k}(U)\) denote the Sobolev space. When \(k = 2\), we write \(H^1(U) := W^{1,2}(U)\). \(W_{0}^{1,k}(U)\) denotes the closure of \(C_{c}^{\infty}(U)\) in \(W^{1,k}(U)\).

\subsection{Setting and Assumptions}
Let \( H \) denote a Hilbert space. Fix constant $\lambda > 1$. Consider the measurable mappings $\mathbf{a}(\omega, x): H\times\mathbb{R}^d \to A$, where $A$ denotes the set of $d$-by-$d$ symmetric matrices in $\mathbb{R}^{d \times d}$. The mappings $\mathbf{a}(\omega, x)$ satisfy the following conditions:
\begin{itemize}
   \item[\textbf{A1}.] \textbf{Uniform Ellipticity and Boundedness}: For all $x \in \mathbb{R}^d$ and all $\omega \in H$, we have
	\begin{equation*}
 \forall \xi \in \mathbb{R}^d,\ 
		|\xi|^2 \leq \xi\cdot \mathbf{a}(\omega, x)\xi \leq \lambda|\xi|^2.
	\end{equation*}

    \item[\textbf{A2}.] \textbf{Lipschitz Continuity}: For all $x_1, x_2 \in \mathbb{R}^d$ and all $\omega \in H$,
   \begin{equation*}
       |\mathbf{a}(\omega, x_1) - \mathbf{a}(\omega, x_2)| \leq \lambda |x_1 - x_2|.
   \end{equation*}
  
   \item[\textbf{A3}.] \textbf{Differentiability}: \(\mathbf{a}(\omega, x)\) is continuously differentiable with respect to the parameter \(\omega\), and for every \(\omega \in H\) and all \(x \in \mathbb{R}^d\), 
\begin{equation*}
    |\partial_\omega \mathbf{a}(\omega, x)| \leq \lambda.
\end{equation*}
Here, $\partial_\omega$ denotes the Fréchet derivative with respect to $\omega$.

\end{itemize}
Similarly, let $\mathbf{b}(\omega, x)$ denote mappings from $H\times\mathbb{R}^d$ to the set of $1$-by-$d$ vectors, which satisfy the following conditions:
\begin{itemize}
    \item[\textbf{A4}.] \textbf{Boundedness}: There exists a constant \( K>0 \) such that for every \( \omega \in H \),
    \begin{equation*}
        \Vert \mathbf{b}(\omega, x)\Vert_{L^{\infty}(\mathbb{R}^d)} \leq K.
    \end{equation*}
    Additionally, we assume that \( \Lambda \in [K^2 + 1, \infty) \).

    \item[\textbf{A5}.] \textbf{Differentiability}: \(\mathbf{b}(\omega, x)\) is continuously differentiable with respect to the parameter \(\omega\), and for every \(\omega \in H\) and all \(x \in \mathbb{R}^d\),
\begin{equation*}
    |\partial_\omega \mathbf{b}(\omega, x)| \leq \lambda,
\end{equation*}
\end{itemize}

Following the setup in \cite{Fischer2021}, let \( B^H_1 \) denote an open unit ball in \( H \). We refer to a measurable function \( \tilde{\omega} : \mathbb{R}^d \rightarrow B^H_1 \) as a parameter field. We denote by \( \Omega \) the space of all parameter fields equipped with the \( L^1_{\text{loc}}(\mathbb{R}^d; H) \) topology. A probability measure \(\mathbb{P}\) is assigned to the space \(\Omega\), and a random parameter field sampled from \(\mathbb{P}\) is denoted by \(\omega_{\epsilon} : \mathbb{R}^d \rightarrow B^H_1\). We make assumptions:
 \begin{itemize}
		\item[\textbf{P1}.]  \textbf{Stationarity}: For all $ y\in \mathbb{R}^d,$ $\omega_\epsilon(\cdot+y)$ coincides with the probability distribution of $\omega_\epsilon(\cdot)$.
		
         \item[\textbf{P2}.] \textbf{Fast Decorrelation}: \( \mathbb{P} \) exhibits fast decorrelation on scales larger than \( \epsilon\), as characterized by the spectral gap condition provided in Definition \ref{def:spectral-gap}. Here, \(\epsilon\) is fixed in \( (0 , 1] \).
\end{itemize}

\begin{definition}[Spectral Gap Inequality] \label{def:spectral-gap}
\cite[Definition 16]{Fischer2021} A probability measure \( \mathbb{P} \) of random fields \( \omega_\epsilon \) satisfies a spectral gap inequality with correlation length \(\epsilon \in (0, 1]\) and constant \(\rho > 0\) if for any random variable \( F = F(\omega_\epsilon) \), the following estimate holds:
\[
\mathbb{E}[|F - \mathbb{E}[F]|^2] \leq \frac{\epsilon^d}{\rho^2} \mathbb{E}\left[\int_{\mathbb{R}^d}\left(\fint_{B_\epsilon(x)} \left|\frac{\partial F}{\partial \omega_\epsilon} \right| \, \mathrm{d}\tilde{x} \right)^2 \mathrm{d}x \right],
\]
where \(\fint_{B_\epsilon(x)}\left|\frac{\partial F}{\partial \omega_\epsilon}\right|\mathrm{d}\tilde{x}\) stands short for
\begin{equation}\label{frech}
      \sup_{\delta \omega_\epsilon} \limsup_{t \to 0}  \frac{|F(\omega_\epsilon+t\delta\omega_\epsilon)-F(\omega_\epsilon)|}{t},
  \end{equation}
and the \(\mathrm{sup}\) runs over all random fields \(\delta\omega_\epsilon: \mathbb{R}^d \to H\) supported in \(B_\epsilon(x)\) with \(\Vert \delta\omega_\epsilon\Vert_{L^\infty(\mathbb{R}^d)} \leq 1\).
\end{definition}

Homogenization studies often use the two-scale expansion method, a crucial technique that informs the definition of the first-order corrector, as detailed in the work of A. Bensoussan, J.-L. Lions, and G. Papanicolaou \cite{bensoussan2011asymptotic}. Our study incorporates lower-order terms, necessitating the computation of expansions for equations with such terms. For simplicity, throughout this work, we denote \(\mathbf{a}^\epsilon := \mathbf{a}(\omega_\epsilon(x), \cdot)\) and \(\mathbf{b}^\epsilon := \mathbf{b}(\omega_\epsilon(x), \cdot)\).

\begin{definition}[First-Order Corrector]\label{first order}
Let \(e_{i}\) denote the \(i\)th standard basis vector of \(\mathbb{R}^d\), and let \(\mathbf{a}^\epsilon\) be the coefficient matrix satisfying assumptions \textbf{A1}-\textbf{A3}, \textbf{P1} and \textbf{P2}. We define the first-order corrector \(\phi_{e_{i}}\) as the unique distributional solution of the corrector problem:
\begin{equation}\label{corrector}
	\nabla \cdot \left(\mathbf{a}^\epsilon \left(e_{i} + \nabla \phi_{e_{i}}\right)\right) = 0 \quad \text{in } \mathbb{R}^d,
\end{equation}
subject to the following conditions:
\begin{itemize}
		\item[1.]  \(\phi_{e_{i}}\) has the regularity \(\phi_{e_{i}} \in H^1_{loc}(\mathbb{R}^d)\) and satisfies \(\fint_{B_1} \phi_{e_{i}}(\omega_\epsilon, \cdot) \, \mathrm{d}x = 0\) almost surely with respect to \(\mathbb{P}\).
  
		\item[2.]  The gradient of the corrector \(\nabla \phi_{e_{i}}\) is stationary in the sense that, for \(\mathbb{P}\)-almost surely and for all \(y \in \mathbb{R}^d\),
  \begin{equation*}
      \nabla \phi_{e_{i}}(\omega_\epsilon, \cdot+y) = \nabla \phi_{e_{i}}(\omega_\epsilon(\cdot+y), \cdot) \quad \text{almost everywhere in } \mathbb{R}^d.
  \end{equation*}
  \item[3.]  The gradient of the corrector \(\nabla \phi_{e_{i}}\) satisfies
  \begin{equation}\label{corrector 3}
      \mathbb{E}\left[ \nabla \phi_{e_{i}}\right] = 0, \quad \text{and} \quad \mathbb{E}\left[ |\nabla \phi_{e_{i}}|^2 \right] < \infty.
  \end{equation}
		\item[4.]  The corrector \(\phi_{e_{i}}\) is sublinear, meaning that \(\mathbb{P}\)-almost surely,
  \begin{equation*}
      \lim_{R \to \infty}\frac{1}{R^2}\fint_{B_R}|\phi_{e_{i}}(\omega_\epsilon, x)|^2 \, \mathrm{d}x = 0.
  \end{equation*}
	\end{itemize}	 
\end{definition}
In our setting, the first-order corrector does not depend on the lower-order terms. However, it is important to note that the definition of the first-order corrector may vary depending on the form of the rescaled lower-order terms. For a comprehensive explanation of the specific calculations, we refer to \cite{wodene}. From first-order corrector \(\phi_{e_i}\) defined in Definition \ref{first order}, we define the homogenized coefficients as follows.
\begin{definition}[Homogenized Coefficients]\label{homogenized a}
	  Let $\mathbf{a}^\epsilon$ and $\mathbf{b}^\epsilon$ be coefficients that satisfy the assumptions \textbf{A1}-\textbf{A5}, \textbf{P1} and \textbf{P2}. We denote the homogenized coefficients by \(\bar{\mathbf{a}} \in \mathbb{R}^{d\times d}\) and \(\bar{\mathbf{b}} \in \mathbb{R}^{1\times d}\), which satisfy the following relations:
	\begin{equation*}
		\left\{
		\begin{aligned}
			&	\bar{\mathbf{a}}e_{i} := \mathbb{E}\left[ \mathbf{a}^\epsilon(e_{i} + \nabla \phi_{e_{i}})\right],	\\
			&	\bar{\mathbf{b}}e_{i} := \mathbb{E}\left[ \mathbf{b}^\epsilon(e_{i} + \nabla \phi_{e_{i}})\right].
		\end{aligned}
		\right.
	\end{equation*}
\end{definition}

\subsection{Main Theorems}
\begin{theorem}\label{R space}
Let $d\geq 3$. Let $\Lambda$, $\mathbf{a}^\epsilon$ and $\mathbf{b}^\epsilon$ be coefficients that satisfy the assumptions \textbf{A1}-\textbf{A5}, \textbf{P1} and \textbf{P2}. Let $\bar{\mathbf{a}}$ and $\bar{\mathbf{b}} $ be as in Definition \ref{homogenized a}.
Let $  u_{0} \in H^2(\mathbb{R}^d; \mathbb{R})\cap W^{1,\infty}(\mathbb{R}^d; \mathbb{R}) $ and $(u^\epsilon)_{\epsilon>0}\in H^1(\mathbb{R}^d; \mathbb{R})$ to be the unique weak solution of
	\begin{equation}\label{q1}
			-\nabla\cdot(\mathbf{a}^\epsilon\nabla u^\epsilon)+\mathbf b^\epsilon \nabla u^\epsilon+\Lambda u^\epsilon=-\nabla\cdot(\mathbf{\bar{a}}\nabla  u_{0})+\mathbf{\bar{b}} \nabla  u_{0}+\Lambda u_0\ \ \mathrm{in}\ \mathbb{R}^d
	\end{equation}
 in a distributional sense.
Then we have
 \begin{equation}\label{quali result}
 \|u^\epsilon - u_0\|_{L^{2d/(d-2)}(\mathbb{R}^d)} \lesssim \mathcal{C}(\omega_\epsilon) \epsilon^{\frac{d}{d+2}} \|\nabla u_0\|_{H^1(\mathbb{R}^d)},
	\end{equation}
       where \(\mathcal{C}(\omega_\epsilon)\) is a random constant such that there exists a constant \(C(d, \lambda, K, \Lambda) > 0\) satisfying \( \mathbb{E}[\mathcal{C}(\omega_\epsilon)] \leq C \).
\end{theorem}
Theorem \ref{R space} addresses the scenario in the entire space \(\mathbb{R}^d\), showing how lower-order terms affect the convergence rate. As dimension \(d\) increases, \(\epsilon^{d/(d+2)}\) approaches \(\epsilon\), indicating that the effect of lower-order terms diminishes with higher dimensions. In contrast, Theorem \ref{qualitative result} focuses on bounded domains.

\begin{theorem}\label{qualitative result}
Let $d\geq 3$. Let coefficients be defined as in Theorem \ref{R space}. Let $\mathcal{O}\subset \mathbb{R}^d $ be a bounded convex Lipschitz domain. Fix $  u_{0} \in H^2(\mathcal{O}; \mathbb{R}) $ and $(u^\epsilon)_{\epsilon>0}\in u_{0}+H^1_0(\mathcal{O}; \mathbb{R})$ to be the unique weak solution of
	\begin{equation*}
			-\nabla\cdot(\mathbf{a}^\epsilon\nabla u^\epsilon)+\mathbf b^\epsilon \nabla u^\epsilon+\Lambda u^\epsilon=-\nabla\cdot(\mathbf{\bar{a}}\nabla  u_{0})+\mathbf{\bar{b}} \nabla  u_{0}+\Lambda u_0\ \ \mathrm{in}\ \mathcal{O}
	\end{equation*}
 in a distributional sense.
Then we have 
	\begin{equation}\label{bounded result}
 \Vert u^{\epsilon}-u_0\Vert_{L^{2}(\mathcal{O})}\lesssim  \mathcal{C}(\omega_\epsilon) \epsilon^{\frac{1}{2}}\Vert \nabla u_0 \Vert_{H^1(\mathcal{O})},
	\end{equation}
 where \(\mathcal{C}(\omega_\epsilon)\) is defined as in Theorem \ref{R space}.
\end{theorem}
It is widely believed that adding lower-order terms does not alter the convergence rate, leading to the conclusion that their effect is negligible and need not be studied. However, Theorem \ref{R space} demonstrates that, in unbounded domains, lower-order terms indeed change the convergence rate, suggesting a notable impact. This may raise questions, which we address by introducing Theorem \ref{qualitative result} to examine the bounded domain case. In bounded domains, even with the inclusion of lower-order terms, the convergence rate remains at \(O(\epsilon^{1/2})\), similar to cases without these terms. This result, however, does not imply that lower-order terms have no effect on convergence. Rather, the influence of boundary conditions is so dominant that it overshadows the impact of lower-order terms. Thus, the influence of lower-order terms on convergence is real, but in bounded domains, it is hidden under the stronger boundary effects.

%% file: 3.Preliminary.tex
\section{Preliminary}\label{sec 2}
\subsection{Flux Corrector and Localized Correctors}
In addition to the corrector \(\phi_{e_i}\), we recall the notion of the flux corrector \(\sigma_i\), as introduced in \cite{GloriaAntoine} and further detailed in \cite{Duerinckx2} and \cite{Fischer2021}. The flux corrector allows us to express the two-scale homogenization error in divergence form, providing a critical tool for analyzing fluctuations in stochastic homogenization.

\begin{definition}[Flux Corrector] \label{flux-corrector}
Let $\mathbf{a}^\epsilon$ and $\bar{\mathbf{a}}$ be defined as in Theorem \ref{R space} and \(\phi_{e_i}\) be as defined in Definition \ref{first order}. Let \(\sigma_i = (\sigma_{ijk})_{i,j,k=1}^{d}\)
 be the 3-tensor defined as the weak solution of
\begin{equation*}
    -\Delta \sigma_{ijk} =  \partial_j q_{ik} - \partial_k q_{ij},
\end{equation*}
 where \(q_i\) denotes the flux of the corrector
\begin{equation*}
    q_i = \mathbf{a}^{\epsilon} (\nabla \phi_i + e_i) - \bar{\mathbf{a}} e_i, \quad q_{ij} := (q_i)_j,
\end{equation*}
with the following conditions:
\begin{itemize}
    \item[1.]  For \(\mathbb{P}\)-almost every realization of the random field \(\omega_\epsilon\), the flux corrector \(\sigma_i(\omega_\epsilon, \cdot)\) has the regularity \(\sigma_i(\omega_\epsilon, \cdot) \in H^1_{\text{loc}}(\mathbb{R}^d; \mathbb{R}^{d \times d})\), satisfies \(\fint_{B_1} \sigma_i(\omega_\epsilon, x) \, \mathrm{d}x = 0\).
    
   \item[2.] The gradient of the flux corrector \(\nabla \sigma_{i}\) is stationary, i.e., for \(\mathbb{P}\)-almost surely and for all \(y \in \mathbb{R}^d\),
    \[
    \nabla \sigma_{i}(\omega_\epsilon, \cdot + y) = \nabla \sigma_{i}(\omega_\epsilon(\cdot + y), \cdot) \quad \text{almost everywhere in } \mathbb{R}^d.
    \]
    
  \item[3.]  The gradient of the flux corrector \(\nabla \sigma_i\) has finite second moments and vanishing expectation, i.e.,
    \[
    \mathbb{E}[\nabla \sigma_{i}] = 0, \quad \mathbb{E}[|\nabla \sigma_i|^2] < \infty.
    \]
    
    \item[4.] The flux corrector \(\sigma_i\) is sublinear at infinity, meaning that \(\mathbb{P}\)-almost surely,
    \[
    \lim_{R \to \infty} \frac{1}{R^2} \fint_{B_R} |\sigma_i(\omega_\epsilon, x)|^2 \, \mathrm{d}x = 0.
    \]
    
    \item[5.]  The flux corrector satisfies the skew-symmetry condition \(\sigma_{ijk} = -\sigma_{ikj}\) \cite[Lemma 1]{BellaP2}.
\end{itemize}
\end{definition}

To utilize the results from \cite{Fischer2021}, particularly the estimates on correctors, we introduce localized correctors. Localized correctors are essential because they overcomes the challenges posed by the original corrector, which is defined over the entire unbounded domain \(\mathbb{R}^d\) and lacks the stationarity property. The stationarity property is crucial for the subsequent proofs in this work. By adopting localized correctors, we can apply the necessary estimates from \cite{Fischer2021} and ensure the stationarity needed for our analysis.

\begin{lemma}[Localized Correctors]\label{Localized correctors}\cite[Lemma 18]{Fischer2021}
Let $\mathbf{a}^\epsilon$ be defined as in Theorem \ref{R space}. Fix \(T \geq 1\). There exists a unique vector field \(\phi_{e_{i}}^T := \phi_{e_{i}}^T(\omega_\epsilon, \cdot)\) which is the weak solution of
\begin{equation*}
-\nabla \cdot (\mathbf{a}^\epsilon(e_{i}+ \nabla \phi^T_{e_i})) + \frac{1}{T}\phi_{e_{i}}^T = 0 \qquad \text{in } \mathbb{R}^d,
\end{equation*}
and the gradient \(\nabla \phi^T_{e_i}\) satisfies
\begin{equation*}
    \mathbb{E}\left[ \nabla \phi^T_{e_i}\right] = 0, \qquad \mathbb{E}\left[ |\nabla \phi^T_{e_i}|^2 \right] < \infty.
\end{equation*}
\end{lemma}

\begin{lemma}[Localized Flux Correctors]\label{Localized flux correctors}\cite[Lemma 18]{Fischer2021} Let $\mathbf{a}^\epsilon$ be defined as in Theorem \ref{R space} and \(\phi^T_{e_i}\) be as defined in Lemma \ref{Localized correctors}. Fix \(T \geq 1\). There exists a unique tensor field \(\sigma^T_{i} := \sigma^T_{i}(\omega_\epsilon, \cdot) \in H^1_{\text{uloc}}(\mathbb{R}^d; \mathbb{R}_\text{skew}^{d \times d})\), which satisfies the following PDE in the distributional sense:
\begin{equation*}
-\Delta \sigma^T_{ijk} + \frac{1}{T}\sigma^T_{ijk} = \partial_j q^T_{ik} - \partial_k q^T_{ij} \qquad \text{in } \mathbb{R}^d,
\end{equation*}
where \(q^T_{i} := \mathbf{a}^\epsilon(\nabla \phi^T_{e_i} + e_i)\) is the localized flux associated with the corrector. 
\end{lemma}

By adding the \(\frac{1}{T} \phi^T_{e_i}\) term, the behavior of the corrector function \(\phi^T_{e_i}\) at infinity is controlled. As \(T\) approaches infinity, the term \(\frac{1}{T} \phi^T_{e_i}\) tends towards zero, causing \(\phi^T_{e_i}\) to increasingly approximate the original corrector function \(\phi_{e_{i}}\), thereby ensuring the effectiveness of the localization method.  The argument for \(\sigma_i^T\) are similar to those for \(\phi^T_{e_i}\). This localized corrector approximates the original correctors in the sense that \cite[Lemma 18]{Fischer2021}
\begin{equation*}
    (\nabla \phi^T_{e_i}, \nabla \sigma^T_{i}) \to  (\nabla \phi_{e_{i}}, \nabla \sigma_{i}) \qquad \text{in } L^2(\Omega \times B_r) \text{ as } T \to \infty, \text{ for any } r > 0.
\end{equation*}

\begin{lemma}\cite[Lemma 2.22]{Stefan2018} \label{Lemma 2.22}
    Suppose the probability measure \(\mathbb{P}\) is stationary, and let \(f(x)\) denote a stationary \(L^1\)-random field. For any open and bounded set \(A \subset \mathbb{R}^d\) and for almost every \(x \in \mathbb{R}^d\), we have
    \begin{equation*}
        \mathbb{E}\left[f(x)\right] = \mathbb{E}\left[\fint_A f(x) \, \mathrm{d}x\right] = \mathbb{E}\left[\fint_{B_1} f(x) \, \mathrm{d}x\right].
    \end{equation*}
\end{lemma}
Since localized correctors satisfies stationarity, we can apply Lemma \ref{Lemma 2.22} to them.

\subsection{Estimates for Correctors}
In the study of quantitative homogenization, the primary focus has been on obtaining precise quantitative estimates for correctors. Key contributions in this area include the work of Gloria and Otto \cite{GloriaAntoine2}, who provided quantitative results for the corrector equation, and the higher-order error estimates in weak norms using second-order correctors by Bella et al. \cite{BellaP}. Additionally, Gu and Mourrat \cite{Gu2016} offered pointwise two-scale expansion estimates under finite-range dependence conditions, similar to the quantitative estimates by Armstrong et al. \cite{armstrong2018quantitative, armstrong2019quantitative}.

In this work, we build on those foundational works, particularly leveraging the quantitative estimates of correctors from \cite{Fischer2021}, where the first-order corrector is analyzed in various \(L^p\) norms. This forms the foundation of our analysis.

\begin{lemma}\cite[Lemma 32]{Fischer2021} \label{lemma32}
    Let \(\epsilon > 0\), \(m \geq 2\), and \(J \geq 0\), with \(d \geq 3\). Let \(u = u(\omega_\epsilon, x)\) be a random field satisfying the estimates
    \begin{equation*}
        \mathbb{E}\left[\left(\fint_{B_\epsilon(x_0)}\left|\nabla u\right|^2 \, \mathrm{d}x\right)^{m/2}\right]^{1/m} \leq J,
    \end{equation*}
    and 
    \begin{equation*}
        \mathbb{E}\left[\left(\int_{\mathbb{R}^d}\nabla u \cdot g \, \mathrm{d}x\right)^m \right]^{1/m} \leq J\left(\frac{\epsilon}{r}\right)^{d/2},
    \end{equation*}
    for all \(r \geq 2\epsilon\), \(x_0 \in \mathbb{R}^d\), and vector fields \(g: \mathbb{R}^d \to \mathbb{R}^d\) supported in \(B_r(x_0)\) that satisfy
    \begin{equation}\label{ggg}
        \left(\fint_{B_r(x_0)}|g|^{2+1/d} \, \mathrm{d}x\right)^{1/(2+1/d)} \leq r^{-d}.
    \end{equation}
   Then, for all \(r \geq \epsilon\) and \(2 \leq p \leq 2d/(d-2)\), the following holds:
    \begin{equation*}
        \mathbb{E}\left[\left(\fint_{B_r(x_0)}\left| u-\fint_{B_r(x_0)}u\right|^p \, \mathrm{d}x\right)^{m/p}\right]^{1/m} \leq C J\epsilon,
    \end{equation*}
    with $C$ only depending on $d, p$.
\end{lemma}

Lemma \ref{lemma32} extends the \(L^2\) norm estimate from \cite{BellaP} to \(L^p\) norms. Specifically, \cite[Theorem 6.1]{BellaP} focuses on \(L^2\) norm estimates under the condition \(\epsilon = 1\) for \(d \geq 3\), using results from \cite[Theorem 1 and Theorem 3]{GloriaAntoine}. This extension to \(L^p\) norms allows us to get the estimates for correctors \(\phi_{e_i}^T\) and \(\sigma_i^T\).

\begin{lemma}
Let \(\phi_{e_i}^T\) be the localized corrector defined in Lemma \ref{Localized correctors}. Then, for any \(x_0 \in \mathbb{R}^d\), there exists a constant \(C(d, \lambda) > 0\) such that the following estimates hold:
\begin{equation}\label{phi2}
    \mathbb{E}\left[\fint_{B_1(x_0)} \left|\phi_{e_i}^T\right|^2 \, \mathrm{d}x\right] \leq C \epsilon^2,
\end{equation}
and
\begin{equation}\label{phi22}
    \mathbb{E}\left[\fint_{B_1(x_0)} \left|\phi_{e_i}^T\right|^{2d/(d-2)} \, \mathrm{d}x\right] \leq C \epsilon^{2d/(d-2)}.
\end{equation}
Similarly, for the localized flux corrector \(\sigma_i^T\) defined in Lemma \ref{Localized flux correctors}, we have
\begin{equation}\label{sigma2}
    \mathbb{E}\left[\fint_{B_1(x_0)} \left|\sigma_i^T\right|^2 \, \mathrm{d}x\right] \leq C \epsilon^2.
\end{equation}
\end{lemma}

\begin{proof}
By applying \cite[Lemma 26 and Lemma 48]{Fischer2021}, we deduce that, for any \(x_0 \in \mathbb{R}^d\), there exists a constant \(C(d,\lambda)<\infty\) such that
\begin{equation*}
  \mathbb{E}\left[\left(\fint_{B_\epsilon(x_0)} \left|\nabla \phi_{e_i}^T\right|^2 \, \mathrm{d}x \right)^{m/2} \right]^{1/m} \leq C.  
\end{equation*}

Let \(F\) denote
\begin{equation*}
    F(\omega_\epsilon) := \int_{\mathbb{R}^d} g \cdot \nabla\phi^T_{e_i} \, \mathrm{d}x,
\end{equation*}
where \(g \in L^{2+1/d}(\mathbb{R}^d;\mathbb{R}^d)\) is a deterministic vector field, satisfying (\ref{ggg}).
From \cite[Lemma 24 and Lemma 26]{Fischer2021}, we have, for \(d \geq 3\) and any \(m \geq 2\),
\begin{equation*}
    \mathbb{E}[|F|^{m}]^{1/m} \leq C(d,\lambda) \left(\frac{\epsilon}{r}\right)^{d/2}.
\end{equation*}
Then, by applying Lemma \ref{lemma32}, for all \(r \geq \epsilon\), \(x_0 \in \mathbb{R}^d\), and \(2 \leq p \leq 2d/(d-2)\), we obtain
\begin{equation*}
    \mathbb{E}\left[\left(\fint_{B_r(x_0)}\left| \phi^T_{e_i}-\fint_{B_r(x_0)} \phi^T_{e_i} \right|^p \, \mathrm{d}x\right)^{m/p}\right]^{1/m} \leq C \epsilon.
\end{equation*}
By using the relation \(\underset{R \to \infty}{\mathrm{lim}} \fint_{B_R(x_0)} \phi^T_{e_i} \, \mathrm{d}x = 0\) for the case \(d \geq 3\), and by setting \(r = 1\) and \(x_0 = 0\), and considering cases \(p = m = 2\) and \(p = m = 2d/(d-2)\), we obtain the estimates (\ref{phi2}) and (\ref{phi22}).

For the flux corrector \(\sigma_i^T\), a similar argument applies, we arrive at estimate (\ref{sigma2}). For a detailed proof, see \cite[Proposition 19]{Fischer2021}.
\end{proof}


%% file: 4.Estimations_for_correctors.tex
\section{Main Estimates}\label{sec 3}

In this section, we focus on estimating the correctors.

\begin{lemma}\label{phi sigma}
Let \(\mathcal{O}\) denote either \(\mathbb{R}^d\) or a bounded set in \(\mathbb{R}^d\). Let \(u_0\) denote the homogenized solution in \(\mathcal{O}\), as defined in Theorem \ref{R space} for the entire space \(\mathbb{R}^d\) or in Theorem \ref{qualitative result} for bounded domains. Let \(\phi_{e_i}\) denote the first-order corrector defined in Definition \ref{first order}, and let \(\sigma_i\) be defined as in Definition \ref{flux-corrector}. We have
\begin{equation}\label{phi2 sigma}
    \Vert \phi_{e_{i}} \nabla(\partial_{i} {u}_0) \Vert^2_{L^{2}(\mathcal{O})}+\Vert \phi_{e_{i}} \partial_{i} {u}_0 \Vert^2_{L^{2}(\mathcal{O})}+\Vert \sigma_i \nabla(\partial_{i} {u}_0) \Vert^2_{L^{2}(\mathcal{O})} \leq \mathcal{C}(\omega_\epsilon)\epsilon^2 \Vert \nabla u_0 \Vert_{H^1(\mathcal{O})}^2,
\end{equation}
and
\begin{equation}\label{2d d-2}
    \Vert  \phi_{e_{i}}  \partial_{i} {u}_0 \Vert_{L^{2d/(d-2)}(\mathbb{R}^d)} \leq \mathcal{C}(\omega_\epsilon)\epsilon \Vert \nabla u_0 \Vert_{L^{2d/(d-2)}(\mathbb{R}^d)},
\end{equation}
 where \(\mathcal{C}(\omega_\epsilon)\) is a random constant such that there exists a constant \(C(d, \lambda)\) satisfying \( \mathbb{E}[\mathcal{C}(\omega_\epsilon)] \leq C \).
\end{lemma}

\begin{proof}
We begin by considering localized corrector \(\phi^T_{e_i}\) as defined in Lemma \ref{Localized correctors}. Notably, it is the stationarity of \(\phi^T_{e_{i}}\) that allows us to invoke Lemma \ref{Lemma 2.22}. By using the Cauchy–Schwarz inequality, Lemma \ref{Lemma 2.22}, and the estimate given in (\ref{phi2}), we have
\begin{align*} 
\mathbb{E}\left[\int_{\mathcal{O}}\left|\phi^T_{e_{i}} \nabla(\partial_{i} {u}_0) \right|^2 \, \mathrm{d}x\right] &\leq \int_{\mathcal{O}} \sum_{i=1}^{d} \mathbb{E}\left[\left|\phi^T_{e_{i}}\right|^2 \right] \sum_{i=1}^{d} \left|\nabla(\partial_{i} {u}_0)\right|^2 \, \mathrm{d}x \nonumber\\
    &= \int_{\mathcal{O}}  \sum_{i=1}^{d} \mathbb{E}\left[\fint_{B_1}\left|\phi^T_{e_{i}}\right|^2 \, \mathrm{d}y \right]\left|\nabla^2 u_0\right|^2 \, \mathrm{d}x\nonumber\\
    &\leq C\epsilon^2 \Vert \nabla^2 u_0 \Vert_{L^2(\mathcal{O})}^2,
\end{align*}
and
\begin{align*}
    \mathbb{E}\left[\int_{\mathcal{O}}\left|\phi^T_{e_{i}} \partial_{i} {u}_0\right|^2 \, \mathrm{d}x\right] \leq C\epsilon^2 \Vert \nabla u_0 \Vert_{L^2(\mathcal{O})}^2.
\end{align*}

Similarly, by applying (\ref{phi22}), we obtain
\begin{align*}
    \mathbb{E}\left[\int_{\mathbb{R}^d}\left|\phi^T_{e_{i}} (\partial_{i} {u}_0) \right|^{2d/(d-2)} \, \mathrm{d}x\right]
    \leq& \int_{\mathbb{R}^d} \sum_{i=1}^{d} \mathbb{E}\left[\fint_{B_1}\left|\phi^T_{e_{i}}\right|^{2d/(d-2)} \, \mathrm{d}y \right]\left|\nabla u_0\right|^{2d/(d-2)} \, \mathrm{d}x\nonumber\\
    \leq& C\epsilon^{2d/(d-2)} \Vert \nabla u_0 \Vert_{L^{2d/(d-2)}(\mathbb{R}^d)}^{2d/(d-2)}.
\end{align*}

The analysis for \(\sigma_i\) follows similarly. Applying (\ref{sigma2}), we get
\begin{align*}
    \mathbb{E}\left[\int_{\mathcal{O}}\left|\sigma^T_i \nabla(\partial_{i} {u}_0) \right|^2 \, \mathrm{d}x\right] &\le \int_{\mathcal{O}} \sum_{i=1}^{d} \mathbb{E}\left[\left|\sigma^T_i\right|^2 \right]\left|\nabla(\partial_{i} {u}_0) \right|^2 \, \mathrm{d}x\nonumber\\
    &= \int_{\mathcal{O}}\sum_{i=1}^{d} \mathbb{E}\left[\fint_{B_1}\left|\sigma^T_i\right|^2 \, \mathrm{d}y \right]\left|\nabla^2 u_0\right|^2 \, \mathrm{d}x\nonumber\\
    &\leq C\epsilon^2 \Vert \nabla^2 u_0 \Vert_{L^2(\mathcal{O})}^2.
\end{align*}
Finally, we take the limit \(T \to \infty\) in these estimates to obtain (\ref{phi2 sigma}) and (\ref{2d d-2}).
\end{proof}

 To extend the analysis to scenarios involving lower-order terms, it is essential to examine the derivatives of \(\phi_{e_i}\) in detail. In the next Lemma, we provide estimates for these derivatives, which are crucial for understanding the behavior of correctors influenced by lower-order terms.

\begin{lemma}\label{lemma 26}
Let \(\phi_{e_i}\) be the first-order corrector as defined in Definition \ref{first order}. For \(x \in \mathbb{R}^d\) and \(\epsilon \in (0, 1]\), we have
\begin{align*}
    \mathbb{E}\left[\fint_{B_{\epsilon}(x)}\left|e_i+\nabla \phi_{e_i}\right|^2 \, \mathrm{d}y\right] \leq C,
\end{align*}
and
\begin{align}\label{3.2}
    \mathbb{E}\left[\fint_{B_\epsilon(x)} \left|\partial_{\omega_\epsilon}\nabla\phi_{e_i}\right|^2 \, \mathrm{d}y\right]
    \leq C,
\end{align}
where \(C< \infty\) is a constant depending only on $d,\lambda$.
\end{lemma}
\begin{proof}
Since $\nabla \phi_{e_i}$ is stationary, by applying Lemma \ref{Lemma 2.22} and from (\ref{corrector 3}), we obtain
\begin{align}\label{stationary}
\mathbb{E}\left[\fint_{B_{\epsilon}(x)}\left|e_i+\nabla \phi_{e_i}\right|^2\mathrm{d}y\right] &= \mathbb{E}\left[\fint_{B_{\epsilon}(x)}\left(|e_i|^2+2e_i \nabla \phi_{e_i}+|\nabla \phi_{e_i}|^2\right) \mathrm{d}y\right]\nonumber\\
&= 1 + \mathbb{E}\left[|\nabla \phi_{e_i}|^2 \right]\nonumber\\
&\le C.
\end{align}

Next we prove (\ref{3.2}). Differentiating both sides of equation (\ref{corrector}) with respect to \(\omega_\epsilon\), we obtain
\begin{equation*}
-\nabla \cdot \left(( \partial_{\omega_\epsilon}\mathbf{a}^\epsilon)\left( e_i+ \nabla \phi_{e_i} \right) +\mathbf{a}^\epsilon\nabla (\partial_{\omega_\epsilon} \phi_{e_i}) \right) = 0.
\end{equation*}
Multiplying by the test function \(\eta\partial_{\omega_\epsilon} \phi_{e_i}\) and integrating over \(B_\epsilon(x)\), we obtain
\begin{equation*}
\fint_{B_\epsilon(x)} \left(-\nabla \cdot \left( ( \partial_{\omega_\epsilon}\mathbf{a}^\epsilon) \left( e_i+ \nabla \phi_{e_i} \right) +\mathbf{a}^\epsilon\nabla (\partial_{\omega_\epsilon} \phi_{e_i}) \right) \right) \eta\partial_{\omega_\epsilon} \phi_{e_i} \, \mathrm{d}y = 0,
\end{equation*}
where \(\eta \in C_0^\infty(B_\epsilon(x))\) is a smooth function with compact support, \(\eta\) equals to 1 inside \(B_\epsilon(x)\) and zero on the boundary. By using integration by parts, we have
\begin{equation}\label{transf}
\fint_{B_\epsilon(x)} \left((\partial_{\omega_\epsilon}\mathbf{a}^\epsilon)\left( e_i+ \nabla \phi_{e_i} \right) \cdot \nabla (\partial_{\omega_\epsilon} \phi_{e_i}) + \mathbf{a}^\epsilon\nabla (\partial_{\omega_\epsilon} \phi_{e_i}) \cdot \nabla (\partial_{\omega_\epsilon} \phi_{e_i})\right) \, \mathrm{d}y = 0.
\end{equation}
Through transformation of (\ref{transf}) and by using assumption \textbf{A3}, we have
\begin{align}\label{transff}
\fint_{B_\epsilon(x)}\mathbf{a}^\epsilon\nabla (\partial_{\omega_\epsilon} \phi_{e_i}) \cdot \nabla (\partial_{\omega_\epsilon} \phi_{e_i}) \, \mathrm{d}y 
= &-\fint_{B_\epsilon(x)} (\partial_{\omega_\epsilon}\mathbf{a}^\epsilon) \left( e_i+ \nabla \phi_{e_i} \right) \cdot \nabla (\partial_{\omega_\epsilon} \phi_{e_i}) \, \mathrm{d}y\nonumber\\
\le  &\left| \fint_{B_\epsilon(x)} (\partial_{\omega_\epsilon}\mathbf{a}^\epsilon) \left( e_i+ \nabla \phi_{e_i} \right) \cdot \nabla (\partial_{\omega_\epsilon} \phi_{e_i}) \, \mathrm{d}y \right| \nonumber\\
\leq  & \lambda \fint_{B_\epsilon(x)} |e_i+ \nabla \phi_{e_i}| |\nabla (\partial_{\omega_\epsilon} \phi_{e_i})| \, \mathrm{d}y.
\end{align}
For the left-hand side of (\ref{transff}), using assumption \textbf{A1}, we have
\begin{equation}\label{transff2}
 \fint_{B_\epsilon(x)} |\nabla (\partial_{\omega_\epsilon} \phi_{e_i})|^2 \, \mathrm{d}y \le \fint_{B_\epsilon(x)}\mathbf{a}^\epsilon\nabla (\partial_{\omega_\epsilon} \phi_{e_i}) \cdot \nabla (\partial_{\omega_\epsilon} \phi_{e_i}) \, \mathrm{d}y.
\end{equation}
Combing (\ref{transff}) and (\ref{transff2}) and applying Young's inequality, we have
\begin{align}\label{3.29}
    \fint_{B_\epsilon(x)} |\nabla (\partial_{\omega_\epsilon} \phi_{e_i})|^2 \, \mathrm{d}y &\leq \lambda\fint_{B_\epsilon(x)} |e_i+ \nabla \phi_{e_i}| |\nabla (\partial_{\omega_\epsilon} \phi_{e_i})| \, \mathrm{d}y \nonumber\\
     &\leq \frac{\lambda}{2}\fint_{B_\epsilon(x)} |e_i+ \nabla \phi_{e_i}|^2 \, \mathrm{d}y + \frac{1}{2}\fint_{B_\epsilon(x)} |\nabla (\partial_{\omega_\epsilon} \phi_{e_i})|^2 \, \mathrm{d}y.
\end{align}
Hence, from (\ref{3.29}), we have
\begin{align}\label{3.30}
    \fint_{B_\epsilon(x)} |\nabla (\partial_{\omega_\epsilon} \phi_{e_i})|^2 \, \mathrm{d}y
     \leq C\fint_{B_\epsilon(x)} |e_i+ \nabla \phi_{e_i}|^2 \, \mathrm{d}y.
\end{align}
Taking the expectation of (\ref{3.30}) and using (\ref{stationary}), we have
\begin{align*}
\mathbb{E}\left[\fint_{B_\epsilon(x)} \left|\partial_{\omega_\epsilon}\nabla\phi_{e_i}\right|^2 \, \mathrm{d}y\right]\le C
       \mathbb{E}\left[\fint_{B_{\epsilon}(x)}\left|e_i+\nabla \phi_{e_i} \right|^2 \, \mathrm{d}y\right] \leq C.
\end{align*}
This complete the proof of (\ref{3.2}).
\end{proof}

Set 
\begin{equation}\label{gamma}
     \Gamma_i := \mathbf{b}^{\epsilon} \left(e_i + \nabla \phi_{e_i}\right) - \bar{\mathbf{b}}e_i,
\end{equation}
where \(\mathbf{b}^{\epsilon}\) and \(\bar{\mathbf{b}}\) are coefficient defined as in Theorem \ref{R space}, \(\phi_{e_i}\) is the first-order corrector defined in Definition \ref{first order}. The investigation of \(\Gamma_i\) is central to our study, as it captures the effect of lower-order terms. By analyzing \(\Gamma_i\), we aim to clarify how these lower-order terms influence the convergence rate in the homogenization process.

\begin{lemma}\label{pi test l2}
Let \(\Gamma_i\) be as in (\ref{gamma}) and let \(\mathcal{O}\) be defined as Lemma \ref{phi sigma}. Then we have
\begin{align}\label{3.32}
      \|\partial_i {u}_0  \nabla \phi_{e_{i}} \|_{L^{2}(\mathcal{O})}
    \leq \mathcal{C} (\omega_\epsilon) \Vert \nabla u_0 \Vert_{L^{2}(\mathcal{O})},
\end{align}
and 
\begin{equation}\label{gama l2 t}
    \Vert \Gamma_i\partial_{i} u_0 \Vert_{L^{2}(\mathcal{O})}
    \leq  \mathcal{C}(\omega_\epsilon) \Vert \nabla u_0 \Vert_{L^2(\mathcal{O})},
\end{equation}
 where \(\mathcal{C}(\omega_\epsilon)\) is a random constant such that there exists a constant \(C(d, \lambda, K) > 0\) satisfying \( \mathbb{E}[\mathcal{C}(\omega_\epsilon)] \leq C \).
\end{lemma}

\begin{proof}
Since the proof for (\ref{3.32}) is the same to (\ref{gama l2 t}), we only prove (\ref{gama l2 t}) here. Taking the expectation, we have
\begin{equation}\label{gama l2}
    \mathbb{E}\left[ \Vert \Gamma_i\partial_{i} u_0 \Vert^2_{L^{2}(\mathcal{O})}\right]
    \le \int_{\mathcal{O}}\left(\sum_{i=1}^{d}\mathbb{E}\left[\Gamma_i^2 \right]\right)|\nabla u_0|^2 \, \mathrm{d}x.
\end{equation}

By Definition \ref{first order}, we know that
\begin{equation*}
 \mathbb{E}\left[|e_i+\nabla\phi_{e_i}|^2\right]<\infty.   
\end{equation*}
From boundedness assumption \textbf{A4}, we have
\begin{align}\label{3.35}
    \mathbb{E}\left[ \Gamma_{i}^2\right]&=\mathbb{E}\left[\left(\mathbf{b}^{\epsilon}\left(e_{i}+\nabla\phi_{e_i}\right)-\bar{\mathbf{b}}e_{i}\right)^2\right]\nonumber\\
    &=\mathbb{E}\left[\left(\mathbf{b}^{\epsilon}\left(e_{i}+\nabla\phi_{e_i}\right)\right)^2\right]-\left(\bar{\mathbf{b}}e_{i}\right)^2\nonumber\\
    &\leq \mathbb{E}\left[\left(\mathbf{b}^{\epsilon}\left(e_{i}+\nabla\phi_{e_i}\right)\right)^2\right]
    \leq C.
\end{align}
Combine (\ref{gama l2}) and (\ref{3.35}), we have
\begin{equation*}
    \mathbb{E}\left[ \Vert \Gamma_i\partial_{i} u_0 \Vert^2_{L^{2}(\mathcal{O})}\right] \leq C \Vert \nabla u_0 \Vert_{L^2(\mathcal{O})}.
\end{equation*}
This completes the proof of (\ref{gama l2 t}).
\end{proof}

\begin{lemma}\label{two-scale} 
Let coefficients be defined as in Theorem \ref{R space}. Let \(u_0\) be as in Lemma \ref{phi sigma},  and let \(\phi_{e_i}\) be the first-order corrector defined in Definition \ref{first order}. Define the test function \(w^\epsilon\) as follows:
\begin{equation}\label{test function}
    w^{\epsilon} = u_0 + \phi_{e_i} \partial_i u_0 .
\end{equation}
Let \(\mathcal{O}\) be defined as Lemma \ref{phi sigma}. Then, in a distributional sense, we have
\begin{equation}\label{tw}
	-\nabla\cdot(\mathbf{a}^\epsilon\nabla w^\epsilon) + \mathbf{b}^\epsilon \nabla w^\epsilon + \Lambda w^\epsilon = -\nabla\cdot(\mathbf{\bar{a}}\nabla  u_{0}) + \mathbf{\bar{b}}\nabla  u_{0} + \Lambda u_0 + \nabla\cdot \mathcal{R} + r_1 + r_2 \quad \text{in} \ \mathcal{O},
\end{equation}
with the residual term \(\mathcal{R}\) and \(r_1\) satisfying
\begin{align}\label{r contral}
    \Vert \mathcal{R} \Vert^2_{L^{2}(\mathcal{O})} + \Vert r_1 \Vert^2_{L^{2}(\mathcal{O})}
    \lesssim \mathcal{C}(\omega_\epsilon) \epsilon^2 \Vert \nabla u_0 \Vert_{H^1(\mathcal{O})}^2,
\end{align}
  where \(\mathcal{C}(\omega_\epsilon)\) is a random constant such that there exists a constant \(C(d, \lambda, K, \Lambda) > 0\) satisfying \( \mathbb{E}[\mathcal{C}(\omega_\epsilon)] \leq C \), and \(r_2 := \Gamma_i \partial_i u_0\) with \(\Gamma_i\) defined in (\ref{gamma}).
\end{lemma}

\begin{proof}
By substituting \(w^{\epsilon}\) into the left-hand side of (\ref{q1}) for \(u^\epsilon\), and recalling $ q_i = \mathbf{a}^{\epsilon} (\nabla \phi_i + e_i) - \bar{\mathbf{a}} e_i$ defined in Definition \ref{flux-corrector}, we obtain
\begin{align*}
   & (-\nabla\cdot\mathbf{a}^{\epsilon}\nabla+\mathbf{b}^\epsilon\nabla+\Lambda)w^{\epsilon}\nonumber\\
    =& -\nabla\cdot\left(\mathbf{a}^{\epsilon}\left(\phi_{e_{i}} \nabla(\partial_{i} u_0)+\left(e_{i}+\nabla\phi_{e_{i}}\right)\partial_{i} u_0\right)\right)+\mathbf{b}^{\epsilon}\left(\phi_{e_{i}} \nabla(\partial_{i} u_0)+\left(e_{i}+\nabla\phi_{e_{i}}\right)\partial_{i} u_0\right)
  +\Lambda(u_0+\phi_{e_{i}} \partial_{i} u_0 )\nonumber\\
    =& -\nabla\cdot\left(\mathbf{a}^{\epsilon}\phi_{e_{i}}\nabla(\partial_{i} u_0)\right) 
    -\nabla\cdot\left(q_i\partial_{i} u_0+\bar{\mathbf{a}}\nabla u_0\right)+\mathbf{b}^{\epsilon}\phi_{e_{i}} \nabla(\partial_{i} u_0)+\Gamma_i\partial_{i} u_0+\bar{\mathbf{b}}\nabla u_0 
  +\Lambda(u_0+\phi_{e_{i}} \partial_{i} u_0 )\nonumber\\
    =& -\nabla\cdot(\mathbf{\bar{a}}\nabla u_{0})+\mathbf{\bar{b}}\nabla u_{0}+\Lambda u_0-\nabla\cdot\left(\mathbf{a}^{\epsilon}\phi_{e_{i}}\nabla(\partial_{i} u_0)+ q_i\partial_{i} u_0\right)+\mathbf{b}^{\epsilon}\phi_{e_{i}} \nabla(\partial_{i} u_0)
  +\Lambda \phi_{e_{i}} \partial_{i} u_0+\Gamma_i\partial_{i} u_0  \nonumber\\
    =& -\nabla\cdot(\mathbf{\bar{a}}\nabla u_{0})+\mathbf{\bar{b}}\nabla u_{0}+\Lambda u_0+\nabla\cdot \mathcal{R}+r_1+r_2,
\end{align*}
where the residual term \(\mathcal{R} = -\mathbf{a}^{\epsilon} \phi_{e_{i}} \nabla(\partial_{i} u_0) - q_i \partial_{i} u_0\), \(r_1 = \mathbf{b}^{\epsilon} \phi_{e_{i}} \nabla(\partial_{i} u_0) + \Lambda \phi_{e_{i}} \partial_{i} u_0 \), and \(r_2 = \Gamma_i \partial_{i} u_0\). 
Recalling the skew symmetry of \(\sigma_i\) in Definition \ref{flux-corrector}, with \(\nabla\cdot q_i = 0\), we have
\begin{align*}
 \nabla\cdot(q_i\partial_{i} u_0) = q_i\cdot\nabla\partial_{i} u_0 &= (\partial_k \sigma_{ijk}) \partial_j\partial_{i} u_0\nonumber\\
 &= \partial_k (\sigma_{ijk}\partial_j\partial_{i} u_0) - \sigma_{ijk}\partial_k \partial_j\partial_{i} u_0\nonumber\\
 &=-\partial_k (\sigma_{ikj}\partial_j\partial_{i} u_0)\nonumber\\
 &=-\nabla\cdot(\sigma_i\nabla\partial_{i} u_0).
\end{align*}
Therefore, \(\mathcal{R}\) can be expressed as \(\mathcal{R} = (\sigma_i - \mathbf{a}^{\epsilon}\phi_{e_{i}})\nabla(\partial_{i} u_0)\). By applying the triangle inequality and from assumption \textbf{A1}, we have
\begin{align}\label{3.51}
    \Vert \mathcal{R} \Vert^2_{L^{2}(\mathcal{O})} &\leq \Vert \sigma_i \nabla(\partial_{i} u_0) \Vert^2_{L^{2}(\mathcal{O})} + \Vert \mathbf{a}^{\epsilon}\phi_{e_{i}} \nabla(\partial_{i} u_0) \Vert^2_{L^{2}(\mathcal{O})}\nonumber\\
    &\le \Vert \sigma_i \nabla(\partial_{i} u_0) \Vert^2_{L^{2}(\mathcal{O})} + \lambda^2 \Vert \phi_{e_{i}} \nabla(\partial_{i} u_0) \Vert^2_{L^{2}(\mathcal{O})}.
\end{align}
On the other hand, By applying the triangle inequality and from boundedness assumption \textbf{A4}, we have
\begin{align}\label{3.52}
    \Vert r_1 \Vert^2_{L^{2}(\mathcal{O})} &\leq \Vert \mathbf{b}^{\epsilon}\phi_{e_{i}} \nabla(\partial_{i} u_0) \Vert^2_{L^{2}(\mathcal{O})} + \Vert \Lambda\phi_{e_{i}} \partial_{i} u_0 \Vert^2_{L^{2}(\mathcal{O})}\nonumber\\
    &\le K^2\Vert \phi_{e_{i}} \nabla(\partial_{i} u_0) \Vert^2_{L^{2}(\mathcal{O})} + \Lambda^2 \Vert  \phi_{e_{i}} \partial_{i} u_0 \Vert^2_{L^{2}(\mathcal{O})}.
\end{align}
By combining (\ref{3.51}) and (\ref{3.52}), and applying Lemma \ref{phi sigma}, we obtain
\begin{align*}
   &\Vert \mathcal{R} \Vert^2_{L^{2}(\mathcal{O})} + \Vert r_1 \Vert^2_{L^{2}(\mathcal{O})}\nonumber\\
   \le & \max\{1, \lambda^2 + K^2, \Lambda^2\} \left( \Vert \sigma_i \nabla(\partial_i u_0) \Vert^2_{L^{2}(\mathcal{O})} + \Vert \phi_{e_i} \nabla(\partial_i u_0) \Vert^2_{L^{2}(\mathcal{O})}+\Vert  \phi_{e_{i}} \partial_{i} u_0 \Vert^2_{L^{2}(\mathcal{O})} \right) \nonumber\\
   \lesssim & \mathcal{C}(\omega_\epsilon) \epsilon^2 \Vert \nabla u_0 \Vert_{H^1(\mathcal{O})}^2,
\end{align*}
 where \(\mathcal{C}(\omega_\epsilon)\) is defined as in Lemma \ref{phi sigma}. Thus, we obtain (\ref{r contral}).
\end{proof}


When dealing with the complexities introduced by lower-order terms, we localize the analysis to smaller regions by decomposing the domain into cubes. This allows us to apply the Poincaré inequality effectively within these localized regions, providing control over the lower-order terms. Moreover, the spectral gap condition enables us to further manage the contributions of these terms within the local grids. This grid-based approach ensures refined estimates, which are crucial for the overall quantitative analysis of the problem.

Consider the cube:
\begin{equation}\label{BOX}
\Box_\iota := \left(-\dfrac{\iota}{2},\dfrac{\iota}{2}\right)^{d}, \quad (\iota\ge \epsilon).
\end{equation}
For simplicity, we denote \(\Box_1\) by \(\Box\). 
By analyzing within these small cubes \(\Box_\iota\), we can effectively break down the contributions from lower-order terms and ensure that the overall estimates remain tractable.
\begin{lemma}\label{poincare lemma}
Let $
\Box_\iota
$ be defined as in (\ref{BOX}). Let \( v(x) \in L^2(\Box_\iota) \) be a function. There exists a constant \(\mu(\iota)\), depending on the domain size \(\iota\) and the dimension \(d\), such that
\begin{equation}\label{eq:first inequality}
    \Vert v - (v)_\iota \Vert_{L^2(\Box_\iota)} \leq \mu(\iota) \Vert \nabla v \Vert_{L^2(\Box_\iota)}.
\end{equation}
Moreover, for \(v(x) = v_\iota(x/\iota)\) with \(x \in \Box_\iota\), we have
\begin{align}\label{eq:second inequality}
    \Vert v - (v)_\iota \Vert_{L^2(\Box_\iota)} \leq \mu(1) \iota \Vert \nabla v \Vert_{L^2(\Box_\iota)}.
\end{align}
\end{lemma}

\begin{proof}
The inequality \eqref{eq:first inequality} follows directly from the Poincaré inequality. 

Now, set \(v(x) = v_\iota(x/\iota)\) for \(x \in \Box_\iota\), where \(v_\iota \in L^2(\Box)\) and \(v \in L^2(\Box_\iota)\). We calculate the average of \(v_\iota(x)\) as follows:
\[
\fint_\Box v_\iota(x) \,\mathrm{d}x = \int_\Box v(\iota x) \,\mathrm{d}x = \iota^{-d} \int_{\Box_\iota} v(x) \,\mathrm{d}x = \fint_{\Box_\iota} v(x) \,\mathrm{d}x.
\]
Thus, we can deduce that
\begin{align*}
   \Vert v - (v)_\iota \Vert_{L^2(\Box_\iota)} 
   = \Vert v_\iota - (v_\iota)_1 \Vert_{L^2(\Box)} \leq \mu(1) \Vert \nabla v_\iota \Vert_{L^2(\Box)} = \mu(1) \iota \Vert \nabla v \Vert_{L^2(\Box_\iota)}.
\end{align*}
This concludes the proof of inequality \eqref{eq:second inequality}.
\end{proof}

\begin{lemma}\label{pi test}
Let \(\Gamma_i\) be as in (\ref{gamma}). For any \(z\in \iota \mathbb{Z}^d\) with \(\iota > \epsilon\), we have
\begin{equation}\label{p test}
    \fint_{z+\Box_\iota} \Gamma_i\partial_{i} u_0 \, \mathrm{d}x
    \leq \mathcal{C}(\omega_\epsilon) \epsilon^{d/2} \iota^{-d} \Vert \nabla u_0\Vert_{L^2(z+\Box_\iota)},
\end{equation}
 where \(\mathcal{C}(\omega_\epsilon)\) is defined as in Lemma \ref{pi test l2} and \(\Box_\iota\) is the cube as defined in (\ref{BOX}).
\end{lemma}

\begin{proof}
Let \(F\) denote
\begin{equation*}
    F(\omega_\epsilon) := \fint_{z+\Box_\iota } \Gamma_i\partial_{i} u_0 \, \mathrm{d}x.
\end{equation*}
By Definition \ref{homogenized a}, we have \(\mathbb{E}[F(\omega_\epsilon)] = 0\). By assumption \textbf{P2}, we get 
\begin{equation}\label{patial omega phi00}
    \mathbb{E}[|F|^{2}] = \mathbb{E}[|F-\mathbb{E}[F]|^{2}]
    \leq C\epsilon^d\mathbb{E}\left[\left|\int_{\mathbb{R}^d}\left(\fint_{B_\epsilon(x)}\left|\frac{\partial F}{\partial \omega_\epsilon}\right| \, \mathrm{d}y \right)^2 \, \mathrm{d}x\right|\right].
\end{equation}
If \(B_\epsilon(x) \cap (z + \Box_\iota) \neq \emptyset\), by using (\ref{frech}) and assumptions \textbf{A4} and \textbf{A5}, we have
\begin{align}\label{derivetive b}
     \fint_{B_\epsilon(x)} \left|\frac{\partial F}{\partial \omega_\epsilon}\right| \, \mathrm{d}y
    = & \sup_{\delta\omega_\epsilon} \limsup_{t \to 0} \frac{|F(\omega_\epsilon + t\delta\omega_\epsilon) - F(\omega_\epsilon)|}{t}\nonumber\\
    = &\sup_{\delta\omega_\epsilon} \limsup_{t \to 0} \frac{\left|\fint_{z + \Box_\iota} \left(\Gamma_i(\omega_\epsilon + t\delta\omega_\epsilon)- \Gamma_i(\omega_\epsilon)  \right) \partial_i u_0 \, \mathrm{d}y \right|}{t}\nonumber\\
    \le & \sup_{\delta\omega_\epsilon} \fint_{z + \Box_\iota} \left|(\partial_{\omega_\epsilon} \mathbf{b}(\omega_\epsilon) ) \left(e_i+ \nabla \phi_{e_i}\right) \delta\omega_\epsilon \partial_i u_0 + \mathbf{b}(\omega_\epsilon) (\partial_{\omega_\epsilon} \nabla \phi_{e_i}) \delta\omega_\epsilon \partial_i u_0\right| \, \mathrm{d}y\nonumber\\
    \leq & \lambda \iota^{-d} \int_{z + \Box_\iota \cap B_\epsilon(x)}  \left|\left(e_i+ \nabla \phi_{e_i}\right) \partial_i u_0\right| \, \mathrm{d}y + K \iota^{-d} \int_{z + \Box_\iota \cap B_\epsilon(x)} |(\partial_{\omega_\epsilon} \nabla \phi_{e_i} ) \partial_i u_0| \, \mathrm{d}y\nonumber\\
    := & T_1+T_2.
\end{align}
Applying the Cauchy–Schwarz inequality and Lemma \ref{lemma 26}, we have 
\begin{align}\label{patial omega phi0}
    \mathbb{E}\left[\left|\int_{\mathbb{R}^d} T_1 ^2 \, \mathrm{d}x\right|\right]
    \lesssim & \int_{\mathbb{R}^d} \sum_{i=1}^{d} \mathbb{E}\left[\iota^{-d}\int_{z+\Box_\iota \cap B_\epsilon(x)} \left|e_i+ \nabla \phi_{e_i}\right|^2 \, \mathrm{d}y\right] \iota^{-d}\int_{z+\Box_\iota \cap B_\epsilon(x)} \left|\nabla u_0\right|^2 \, \mathrm{d}y \, \mathrm{d}x \nonumber\\
    \lesssim & \int_{\mathbb{R}^d}  \sum_{i=1}^{d} \mathbb{E}\left[\fint_{B_\epsilon(x)} \left|e_i+ \nabla \phi_{e_i}\right|^2 \, \mathrm{d}y\right] \iota^{-d}\int_{z+\Box_\iota \cap B_\epsilon(x)} \left|\nabla u_0\right|^2 \, \mathrm{d}y \, \mathrm{d}x \nonumber\\
    \lesssim  &  \int_{B_{3/2\iota}(z)}\int_{z+\Box_\iota} \iota^{-d} \left|\nabla u_0\right|^2 \, \mathrm{d}y \, \mathrm{d}x\nonumber\\
    \lesssim  &  \iota^{-2d}\Vert \nabla u_0 \Vert^2_{L^2(z+\Box_\iota)}.
\end{align}
Indeed, for the last two inequalities in (\ref{patial omega phi0}), we examine the relationship between the regions \(z + \Box_\iota\) and \(B_\epsilon(x)\). Observe that if \(|z - x| \leq \frac{1}{2} \iota + \epsilon\), then the intersection \(z + \Box_\iota \cap B_\epsilon(x) \neq \emptyset\). This indicates that the region \(B_\epsilon(x)\) overlaps with \(z + \Box_\iota\), and hence the integration region is non-empty. Note that since \(\iota > \epsilon\), the inequality \(|z - x| \leq \frac{1}{2} \iota + \epsilon\) implies that \(x\) lies within a ball centered at \(z\) with radius \(\frac{1}{2}\iota + \epsilon\). This ball can be denoted as \(B_{\frac{1}{2}\iota + \epsilon}(z)\). Furthermore, since \(\frac{1}{2}\iota + \epsilon \leq \frac{3}{2}\iota\), it follows that \(B_{\frac{1}{2}\iota + \epsilon}(z) \subset B_{\frac{3}{2}\iota}(z)\).
Similarly, applying the Cauchy–Schwarz inequality and Lemma \ref{lemma 26} again, we have
\begin{align}\label{patial omega phi1}
    \mathbb{E}\left[\left|\int_{\mathbb{R}^d} T_2 ^2 \, \mathrm{d}x\right|\right]
    \lesssim & \int_{B_{3/2\iota}(z)}  \sum_{i=1}^{d}  \mathbb{E}\left[\fint_{B_\epsilon(x)} \left|\partial_{\omega_\epsilon} \nabla \phi_{e_i}\right|^2 \, \mathrm{d}y\right] \int_{z+\Box_\iota} \iota^{-d} \left| \nabla u_0\right|^2 \, \mathrm{d}y \, \mathrm{d}x \nonumber\\
    \lesssim &  \iota^{-2d}\sum_{i=1}^{d} \Vert \partial_i u_0\Vert^2_{L^2(z+\Box_\iota)} \nonumber\\
    = & C \iota^{-2d}\Vert \nabla u_0\Vert^2_{L^2(z+\Box_\iota)},
\end{align}
where $C$ only depend on $d, \lambda, K$. For the first inequality in (\ref{patial omega phi1}), similar to the previous case in (\ref{patial omega phi0}), the integral region can be restricted to \(B_{3/2\iota}(z)\) due to the relationship between \(z\), \(\iota\), and \(\epsilon\). 
Substituting (\ref{derivetive b})-(\ref{patial omega phi1}) into (\ref{patial omega phi00}), we obtain
\begin{equation*}
    \mathbb{E}[|F|^{2}] \leq C \epsilon^d \iota^{-2d}\Vert \nabla u_0\Vert^2_{L^2(z+\Box_\iota)}.
\end{equation*}
Consequently, by applying the inequality for the expectation, we have
\begin{equation*}
    \left(\fint_{z+\Box_\iota } \Gamma_i\partial_{i} u_0 \, \mathrm{d}x \right)^2
    \leq \mathcal{C}(\omega_\epsilon) \epsilon^d \iota^{-2d} \Vert \nabla u_0 \Vert_{L^2(z+\Box_\iota)}^2,
\end{equation*}
where \(\mathcal{C}(\omega_\epsilon)\) is a random constant that satisfies \(\mathbb{E}[\mathcal{C}(\omega_\epsilon)] \leq C\) for some constant \(C(d, \lambda, K) > 0\). This completes the proof of (\ref{p test}).
\end{proof}

%% file: 5.RSPACE.tex
\section{Proofs of Main Theorems}\label{sec 4}

In the previous sections, we developed crucial preparatory results, including the corrector estimates and the two-scale expansion, which are fundamental for analyzing the behavior of the solutions. These estimates, particularly those concerning the lower-order terms, are central to our approach. In the proof of the main theorem, we address the challenges posed by these lower-order terms by introducing a localization technique. 

With these foundational results in place, we now turn to the proof of the main theorems, wherein we derive the desired convergence rate.

\begin{proof}[Proof of Theorem \ref{R space}]
By applying the triangle inequality and recalling the test function \( w^{\epsilon} = u_0 + \phi_{e_i} \partial_i u_0 \) defined in (\ref{test function}), along with the Sobolev embedding theorem \cite[Theorem 9.9]{brezis2011functional}, we obtain
\begin{align}\label{r left left0}
		\Vert u^{\epsilon}-u_0\Vert_{L^{2d/(d-2)}(\mathbb{R}^d)}&=\Vert u^{\epsilon}-w^{\epsilon}+w^{\epsilon}-u_0\Vert_{L^{2d/(d-2)}(\mathbb{R}^d)}\nonumber\\
       &\le \Vert u^{\epsilon}-w^{\epsilon}\Vert_{L^{2d/(d-2)}(\mathbb{R}^d)}+\Vert \phi_{e_{i}} \partial_{i} {u}_0 \Vert_{L^{2d/(d-2)}(\mathbb{R}^d)} \nonumber\\
     &\lesssim \Vert \nabla(u^{\epsilon}-w^{\epsilon})\Vert_{L^{2}(\mathbb{R}^d)}+\Vert \phi_{e_{i}} \partial_{i} {u}_0\Vert_{L^{2d/(d-2)}(\mathbb{R}^d)}.
	\end{align}    

From equations (\ref{q1}) and (\ref{tw}), subtracting one from the other yields:
\begin{equation}\label{sub}
  -\nabla \cdot (\mathbf{a}^\epsilon \nabla (u^\epsilon - w^\epsilon)) + \mathbf{b}^\epsilon \nabla (u^\epsilon - w^\epsilon) + \Lambda (u^\epsilon - w^\epsilon) = -\nabla \cdot \mathcal{R} - r_1 - r_2.  
\end{equation}

From (\ref{sub}) and by using weak formulation and Young's inequality, we have 
 \begin{align}\label{zui3}
& \int_{\mathbb{R}^d}({u}^\epsilon- w^{\epsilon})\left(-\nabla\cdot\mathbf{a}^{\epsilon}\nabla+\mathbf{b}^{\epsilon}\nabla+\Lambda\right)  ({u}^\epsilon-w^{\epsilon})\,\mathrm{d}x\nonumber\\
=& \int_{\mathbb{R}^d}({u}^\epsilon-w^{\epsilon})(-\nabla\cdot \mathcal{R}-r_1-r_2) \,\mathrm{d}x\nonumber\\
=&\int_{\mathbb{R}^d}\left(\mathcal{R}\nabla(u^\epsilon-w^{\epsilon})-r_1(u^\epsilon-w^{\epsilon})-r_2(u^\epsilon-w^{\epsilon}) \right)\,\mathrm{d}x\nonumber\\
\le & \Vert \mathcal{R} \Vert^2_{L^{2}(\mathbb{R}^d)}+\frac{1}{4}\Vert \nabla({u}^\epsilon- w^{\epsilon}) \Vert^2_{L^{2}(\mathbb{R}^d)}+ \frac{1}{4(\Lambda-K^2-1/2)}\Vert r_1 \Vert^2_{L^{2}(\mathbb{R}^d)}\nonumber\\
&+(\Lambda-K^2-1/2)\Vert {u}^\epsilon- w^{\epsilon} \Vert^2_{L^{2}(\mathbb{R}^d)}+\left|\int_{\mathbb{R}^d} r_2
  (u^\epsilon-w^{\epsilon})\,\mathrm{d}x\right|.
\end{align}
Now we discuss the last term on the right hand side of (\ref{zui3}).
Let $\iota>0$. Recall that $
\Box_\iota:
	=\left(-\iota/2,\iota/2\right)^{d}
$ defined in (\ref{BOX}) and $r_2:=\Gamma_i\partial_{i} u_0$ defined in Lemma \ref{two-scale}. Denote $(u^\epsilon-w^{\epsilon})_\iota:=\fint_{z+\Box_\iota} (u^\epsilon-w^{\epsilon})$. We decompose the integral as \(\int_{\mathbb{R}^d} \cdot \,\mathrm{d}x = \sum\limits_{z\in \iota \mathbb{Z}^d} \int_{z+\Box_\iota} \cdot \,\mathrm{d}x\).
By applying the triangle inequality, followed by the Cauchy–Schwarz inequality and the Poincaré inequality, we obtain
\begin{align}\label{pi}
   &  \left| \int_{\mathbb{R}^d} r_2
  (u^\epsilon-w^{\epsilon})\,\mathrm{d}x \right| \nonumber\\
  \le &\left| \sum\limits_{z\in \iota \mathbb{Z}^d} \int_{z+\Box_\iota} \Gamma_i\partial_{i} u_0
  ((u^\epsilon-w^{\epsilon})-(u^\epsilon-w^{\epsilon})_\iota) \,\mathrm{d}x \right| + \left|\sum\limits_{z\in \iota \mathbb{Z}^d} \int_{z+\Box_\iota} \Gamma_i\partial_{i} u_0
  (u^\epsilon-w^{\epsilon})_\iota \,\mathrm{d}x \right| \nonumber\\
  \leq & \left| \sum\limits_{z\in \iota \mathbb{Z}^d} \Vert \Gamma_i\partial_{i} u_0 \Vert_{L^{2}(z+\Box_\iota)}\Vert (u^\epsilon-w^{\epsilon})-(u^\epsilon-w^{\epsilon})_\iota \Vert_{L^{2}(z+\Box_\iota)}  \right| +\left| \sum\limits_{z\in \iota \mathbb{Z}^d} (u^\epsilon-w^{\epsilon})_\iota\int_{z+\Box_\iota} \Gamma_i\partial_{i} u_0 \,\mathrm{d}x \right|
 \nonumber\\ 
  \leq & \left| \sum\limits_{z\in \iota \mathbb{Z}^d} \Vert \Gamma_i\partial_{i} u_0 \Vert_{L^{2}(z+\Box_\iota)} \mu(1) \iota \Vert\nabla(u^\epsilon-w^{\epsilon}) \Vert_{L^{2}(z+\Box_\iota)}  \right| +\left| \sum\limits_{z\in \iota \mathbb{Z}^d}(u^\epsilon-w^{\epsilon})_\iota\int_{z+\Box_\iota} \Gamma_i\partial_{i} u_0 \,\mathrm{d}x  \right|\nonumber\\ 
   := & T_1+T_2,
\end{align}
where $ \mu(1)$ is a constant defined in Lemma \ref{poincare lemma} and the last inequality follows from Lemma \ref{poincare lemma}.
By using Young's inequality, we have
\begin{align}\label{pi term 1}
  T_1
  &\leq  \sum\limits_{z\in \iota \mathbb{Z}^d} \Vert \Gamma_i\partial_{i} u_0 \Vert^2_{L^{2}(z+\Box_\iota)}(\mu(1) \iota)^2 +\sum\limits_{z\in \iota \mathbb{Z}^d}\frac{1}{4}
  \Vert \nabla({u}^\epsilon- w^{\epsilon}) \Vert^2_{L^{2}(z+\Box_\iota)}\nonumber\\
  &=\Vert \Gamma_i\partial_{i} u_0 \Vert^2_{L^{2}(\mathbb{R}^d)}(\mu(1) \iota)^2 +\frac{1}{4}
  \Vert \nabla({u}^\epsilon- w^{\epsilon}) \Vert^2_{L^{2}(\mathbb{R}^d)}.
\end{align}
By using the Cauchy–Schwarz inequality and Young's inequality, we have
\begin{align}\label{pi term2}
     T_2=&\left| \sum\limits_{z\in \iota \mathbb{Z}^d} \left(\fint_{z+\Box_\iota} (u^\epsilon-w^{\epsilon}) \,\mathrm{d}x \int_{z+\Box_\iota} \Gamma_i\partial_{i} u_0 \,\mathrm{d}x
  \right)\right| \nonumber\\
  =&\left| \sum\limits_{z\in \iota \mathbb{Z}^d} \left(\int_{z+\Box_\iota} (u^\epsilon-w^{\epsilon})\iota^{-d} \,\mathrm{d}x \int_{z+\Box_\iota} \Gamma_i\partial_{i} u_0 \,\mathrm{d}x
  \right)\right| \nonumber\\
  \leq & \left(\sum\limits_{z\in \iota \mathbb{Z}^d}\iota^{-d}\left(\int_{z+\Box_\iota} (u^\epsilon-w^{\epsilon})\,\mathrm{d}x \right)^2\right)^{1/2}\left(\sum\limits_{z\in \iota \mathbb{Z}^d}\iota^{-d} \left(\int_{z+\Box_\iota} \Gamma_i\partial_{i} u_0 \,\mathrm{d}x \right)^2\right)^{1/2}\nonumber\\
  \leq &\left(\sum\limits_{z\in \iota \mathbb{Z}^d}\iota^{-d} \iota^{d} \Vert {u}^\epsilon- w^{\epsilon} \Vert^2_{L^{2}(z+\Box_\iota)} \right)^{1/2} \left(\sum\limits_{z\in \iota \mathbb{Z}^d}\iota^{-d} \left(\int_{z+\Box_\iota} \Gamma_i\partial_{i} u_0 \,\mathrm{d}x \right)^2\right)^{1/2}\nonumber\\
  \leq & \frac{1}{2}\Vert {u}^\epsilon- w^{\epsilon} \Vert^2_{L^{2}(\mathbb{R}^d)}+\frac{1}{2} \sum\limits_{z\in \iota \mathbb{Z}^d}\iota^{-d} \left(\int_{z+\Box_\iota} \Gamma_i\partial_{i} u_0 \,\mathrm{d}x \right)^2.
\end{align}
From (\ref{zui3})-(\ref{pi term2}), we have
 \begin{align}\label{78}
& \int_{\mathbb{R}^d}({u}^\epsilon- w^{\epsilon})\left(-\nabla\cdot\mathbf{a}^{\epsilon}\nabla+\mathbf{b}^{\epsilon}\nabla+\Lambda\right)  ({u}^\epsilon-w^{\epsilon})\,\mathrm{d}x\nonumber\\
\le &\Vert \mathcal{R} \Vert^2_{L^{2}(\mathbb{R}^d)}+\frac{1}{2}\Vert \nabla({u}^\epsilon- w^{\epsilon}) \Vert^2_{L^{2}(\mathbb{R}^d)}+ \frac{1}{4(\Lambda-K^2-1/2)}\Vert r_1 \Vert^2_{L^{2}(\mathbb{R}^d)}+(\Lambda-K^2)\Vert {u}^\epsilon- w^{\epsilon} \Vert^2_{L^{2}(\mathbb{R}^d)}\nonumber\\
&+ \Vert \Gamma_i\partial_{i} u_0 \Vert^2_{L^{2}(\mathbb{R}^d)}(\mu(1) \iota)^2 +\frac{1}{2} \sum\limits_{z\in \iota \mathbb{Z}^d}\iota^{-d} \left(\int_{z+\Box_\iota} \Gamma_i\partial_{i} u_0 \,\mathrm{d}x \right)^2.
\end{align}
By applying the Cauchy–Schwarz inequality and Young's inequality, and using the uniform ellipticity assumption \textbf{A1} and the boundedness assumption \textbf{A4}, while noting that \(\Lambda \geq K^2 + 1\), we obtain
  \begin{align}\label{Lam}
&\int_{\mathbb{R}^d} (u^\epsilon - w^\epsilon) \left(-\nabla\cdot\mathbf{a}^{\epsilon}\nabla + \mathbf{b}^{\epsilon}\nabla + \Lambda\right) (u^\epsilon - w^\epsilon) \,\mathrm{d}x\nonumber\\
 =& \int_{\mathbb{R}^d} \nabla (u^\epsilon - w^\epsilon) \cdot \mathbf{a}^{\epsilon}\nabla (u^\epsilon - w^\epsilon)\,\mathrm{d}x + \int_{\mathbb{R}^d} \mathbf{b}^{\epsilon}\nabla (u^\epsilon - w^\epsilon) \cdot (u^\epsilon - w^\epsilon) + \Lambda (u^\epsilon - w^\epsilon)^2  \,\mathrm{d}x \nonumber \\ 
  \geq& \Vert \nabla (u^\epsilon - w^\epsilon) \Vert^{2}_{L^{2}(\mathbb{R}^d)} - \left|\int_{\mathbb{R}^d} \mathbf{b}^{\epsilon}\nabla (u^\epsilon - w^\epsilon) \cdot (u^\epsilon - w^\epsilon) \,\mathrm{d}x\right| + \Lambda \Vert (u^\epsilon - w^\epsilon) \Vert^{2}_{L^{2}(\mathbb{R}^d)} \nonumber \\ 
  \geq& \Vert \nabla (u^\epsilon - w^\epsilon) \Vert^{2}_{L^{2}(\mathbb{R}^d)} - K \Vert \nabla (u^\epsilon - w^\epsilon) \Vert_{L^{2}(\mathbb{R}^d)} \Vert (u^\epsilon - w^\epsilon) \Vert_{L^{2}(\mathbb{R}^d)}+ \Lambda \Vert (u^\epsilon - w^\epsilon) \Vert^{2}_{L^{2}(\mathbb{R}^d)} \nonumber\\
  \geq& \Vert \nabla (u^\epsilon - w^\epsilon) \Vert^{2}_{L^{2}(\mathbb{R}^d)} - \frac{1}{4} \Vert \nabla (u^\epsilon - w^\epsilon) \Vert^{2}_{L^{2}(\mathbb{R}^d)} + (\Lambda - K^2) \Vert (u^\epsilon - w^\epsilon) \Vert^{2}_{L^{2}(\mathbb{R}^d)} \nonumber\\
  \geq& \frac{3}{4} \Vert \nabla (u^\epsilon - w^\epsilon) \Vert^{2}_{L^{2}(\mathbb{R}^d)} + (\Lambda - K^2) \Vert (u^\epsilon - w^\epsilon) \Vert^{2}_{L^{2}(\mathbb{R}^d)}.
\end{align}
Combining (\ref{78}) with (\ref{Lam}), we have
\begin{align}\label{r ll}
     \Vert \nabla({u}^\epsilon- w^{\epsilon}) \Vert^2_{L^{2}(\mathbb{R}^d)} 
  \lesssim \Vert \mathcal{R} \Vert^2_{L^{2}(\mathbb{R}^d)}+ \Vert r_1 \Vert^2_{L^{2}(\mathbb{R}^d)}+\Vert \Gamma_i\partial_{i} u_0 \Vert^2_{L^{2}(\mathbb{R}^d)} \iota^2+\sum\limits_{z\in \iota \mathbb{Z}^d}\iota^{d} \left(\fint_{z+\Box_\iota} \Gamma_i\partial_{i} u_0 \,\mathrm{d}x \right)^2.
\end{align}
Now combining (\ref{r contral}), Lemma \ref{pi test l2}, Lemma \ref{pi test} with (\ref{r ll}), we have
\begin{align}\label{r lll}
   & \Vert \nabla({u}^\epsilon- w^{\epsilon}) \Vert^2_{L^{2}(\mathbb{R}^d)} \nonumber\\
   \lesssim & \mathcal{C}(\omega_\epsilon)\epsilon^2 \Vert \nabla u_0 \Vert_{H^1(\mathbb{R}^d)}^2+\mathcal{C}(\omega_\epsilon) \Vert \nabla u_0 \Vert^2_{L^2(\mathbb{R}^d)}\iota^2
   +\sum\limits_{z\in \iota \mathbb{Z}^d}\iota^{d}\mathcal{C}(\omega_\epsilon) \epsilon^d \iota^{-2d} \Vert \nabla u_0 \Vert_{L^2(z+\Box_\iota)}^2\nonumber\\
   \lesssim & \mathcal{C}(\omega_\epsilon)\epsilon^2 \Vert \nabla u_0 \Vert_{H^1(\mathbb{R}^d)}^2+\mathcal{C}(\omega_\epsilon) \Vert \nabla u_0 \Vert^2_{L^2(\mathbb{R}^d)}\iota^2 +\mathcal{C}(\omega_\epsilon) \epsilon^d \iota^{-d}\Vert \nabla u_0 \Vert_{L^2(\mathbb{R}^d)}^2\nonumber\\
   \lesssim &\mathcal{C}(\omega_\epsilon) (\epsilon+\iota+\epsilon^{d/2} \iota^{-d/2})^2 \Vert \nabla u_0 \Vert^2_{H^1(\mathbb{R}^d)}.
\end{align}

Combining (\ref{r left left0}) and (\ref{r lll}) with Lemma \ref{phi sigma}, we obtain
\begin{align}\label{last}
		 \Vert u^{\epsilon}-u_0\Vert_{L^{2d/(d-2)}(\mathbb{R}^d)}
  \lesssim & \mathcal{C}(\omega_\epsilon) (\epsilon+\iota+\epsilon^{d/2} \iota^{-d/2})\Vert \nabla u_0 \Vert_{H^1(\mathbb{R}^d)}+\Vert  \phi_{e_{i}} \partial_{i} {u}_0\Vert_{L^{2d/(d-2)}(\mathbb{R}^d)}\nonumber\\
  \lesssim & \mathcal{C}(\omega_\epsilon) (\epsilon+\iota+\epsilon^{d/2} \iota^{-d/2})\Vert \nabla u_0 \Vert_{H^1(\mathbb{R}^d)}+\mathcal{C}(\omega_\epsilon)\epsilon \Vert \nabla u_0 \Vert_{L^{2d/(d-2)}(\mathbb{R}^d)}.
	\end{align} 
Given that \(\iota \geq \epsilon\), the dominant contributions arise from the terms \(\iota\) and \(\epsilon^{d/2} \iota^{-d/2}\). To minimize the expression \(\epsilon + \iota + \epsilon^{d/2} \iota^{-d/2}\), it is essential to balance the terms \(\iota\) and \(\epsilon^{d/2} \iota^{-d/2}\), which means $\iota = \epsilon^{d/2} \iota^{-d/2}.$ Then, we get $\iota = \epsilon^{d/d+2}$. Because \( d/d+2 < 1\) and \(\epsilon < 1\), we obtain from (\ref{last}) that
\begin{align}\label{sob}
\|u^\epsilon - u_0\|_{L^{2d/(d-2)}(\mathbb{R}^d)} \lesssim \mathcal{C}(\omega_\epsilon) \epsilon^{\frac{d}{d+2}} \|\nabla u_0\|_{H^1(\mathbb{R}^d)}+\mathcal{C}(\omega_\epsilon)\epsilon \Vert \nabla u_0 \Vert_{L^{2d/(d-2)}(\mathbb{R}^d)}.
\end{align}
Given that \(\Vert \nabla u_0 \Vert_{L^{2d/(d-2)}(\mathbb{R}^d)} \leq C \Vert \nabla u_0 \Vert_{H^1(\mathbb{R}^d)}\) by the Sobolev embedding theorem, we can bound the second term on the right hand side of (\ref{sob}). Then, we have
\begin{equation*}  
\|u^\epsilon - u_0\|_{L^{2d/(d-2)}(\mathbb{R}^d)} \lesssim \mathcal{C}(\omega_\epsilon) \left( \epsilon^{\frac{d}{d+2}} + \epsilon \right) \|\nabla u_0\|_{H^1(\mathbb{R}^d)}.
\end{equation*}
Since \( \epsilon^{\frac{d}{d+2}} > \epsilon \) for \( d \geq 3 \) and \(\epsilon < 1\), the dominant term is \( \epsilon^{\frac{d}{d+2}} \), which allows us to further simplify the expression to
\begin{equation*}
\|u^\epsilon - u_0\|_{L^{2d/(d-2)}(\mathbb{R}^d)} \lesssim \mathcal{C}(\omega_\epsilon) \epsilon^{\frac{d}{d+2}} \|\nabla u_0\|_{H^1(\mathbb{R}^d)}.
\end{equation*}
This completes the proof of (\ref{quali result}).
\end{proof}

%% file: 6.bounded_domain.tex
For a bounded domain, the proof follows a similar approach to that for the entire space. However, there are significant differences in the treatment of test functions between the two cases. Specifically, for a bounded domain, a smooth cutoff function is introduced to mitigate the effects near the boundary, ensuring that the test functions exhibit appropriate behavior. This allows all integrals to remain well-defined, particularly near the boundary. Furthermore, to handle boundary contributions, we employ the following trace-type estimate \cite[(79)]{Fischer2021}:
\begin{equation}\label{trace type}
    \int_{(\partial \mathcal{O})_\tau} |\nabla v|^2 \, \mathrm{d}x \lesssim \tau \|\nabla v\|^2_{H^1(\mathcal{O})},
\end{equation}
for any \(v \in H^1(\mathcal{O})\) and \(\tau \ge \epsilon\). 

\begin{proof}[Proof of Theorem \ref{qualitative result}]
From (\ref{test function}), and by applying the triangle inequality and Poincar\'e inequality, we have
\begin{align}\label{r left left}
		\Vert u^{\epsilon}-u_0\Vert_{L^{2}(\mathcal{O})}&=\Vert u^{\epsilon}-w^{\epsilon}+w^{\epsilon}-u_0\Vert_{L^{2}(\mathcal{O})}\nonumber\\
     &\le \Vert u^{\epsilon}-w^{\epsilon}\Vert_{L^{2}(\mathcal{O})}+\Vert  \phi_{e_{i}} \partial_{i} {u}_0\Vert_{L^{2}(\mathcal{O})} \nonumber\\
     &\lesssim \Vert \nabla(u^{\epsilon}-w^{\epsilon})\Vert_{L^{2}(\mathcal{O})}+\Vert \phi_{e_{i}} \partial_{i} {u}_0\Vert_{L^{2}(\mathcal{O})}.
	\end{align}    

Set the function $f$ as follows,
\begin{equation}\label{eta0}
	f=(1-\eta_{r})(\partial_{i} {u}_0)\phi_{e_{i}},
\end{equation}
where $ \eta_{r} $ is a smooth cutoff function that satisfies
	\begin{equation}\label{eta}
		\left\{
		\begin{aligned}
			&  0<\eta_{r} <1 \quad  \mathrm{in}\ (\partial\mathcal{O})_\tau,\\
			& \eta_{r} \equiv1  \quad \mathrm{in} \ \mathcal{O}\setminus(\partial\mathcal{O})_\tau,\\
			& \eta_{r} \equiv0\quad  \mathrm{on}\ \partial\mathcal{O},\\
	        & \left|\nabla\eta_{r} \right|\lesssim \tau^{-1},
		\end{aligned}
		\right.
	\end{equation}
where $\tau\geq\epsilon $ and $ (\partial\mathcal{O})_\tau:=\{x\in\mathcal{O}:\mathrm{dist}(x,\partial \mathcal{O})<\tau\} $.
Now let $u={u}^\epsilon- w^{\epsilon}$. Similar to (\ref{Lam}), for $ (u+f)\in H_0^1(\mathcal{O})$, we have
  \begin{align}\label{new Lam}
&\frac{3}{4} \Vert \nabla u\Vert^{2}_{L^{2}(\mathcal{O})}+(\Lambda-K^2) \Vert u \Vert^2_{L^{2}(\mathcal{O})} \nonumber \\ 
\leq& \int_{\mathcal{O}} (u+f) \left(-\nabla\cdot\mathbf{a}^{\epsilon}\nabla+\mathbf{b}^{\epsilon}\nabla+\Lambda\right)u\,\mathrm{d}x - \int_{\mathcal{O}} \nabla f\cdot\mathbf{a}^{\epsilon}\nabla u+ f \mathbf{b}^{\epsilon}\nabla u+f \Lambda u \, \mathrm{d}x\nonumber \\ 
\leq& \left|\int_{\mathcal{O}} (u+f) \left(-\nabla\cdot\mathbf{a}^{\epsilon}\nabla+\mathbf{b}^{\epsilon}\nabla+\Lambda\right)u\,\mathrm{d}x \right|+ \left|\int_{\mathcal{O}} \nabla f\cdot\mathbf{a}^{\epsilon}\nabla u+ f \mathbf{b}^{\epsilon}\nabla u+f \Lambda u \, \mathrm{d}x\right|\nonumber \\ 
:= & T_1+T_2.
	\end{align}
 Applying Young's inequality, we have
 \begin{align*}
T_2
\le &\int_{\mathcal{O}}  \lambda |\nabla f| |\nabla u|+K| f||\nabla u|+\Lambda |fu|  \, \mathrm{d}x \nonumber\\
\le & 4\lambda^2 \|\nabla f\|^2_{L^{2}(\mathcal{O})}+\frac{1}{8} \|\nabla u\|^2_{L^{2}(\mathcal{O})}+4K^2\| f\|^2_{L^{2}(\mathcal{O})}+\Lambda^2 \|f\|^2_{L^{2}(\mathcal{O})}+\frac{1}{4}\|u\|^2_{L^{2}(\mathcal{O})}.
\end{align*}
Similar to (\ref{zui3}), form (\ref{q1}) and (\ref{tw}), by using weak formulation and Young's inequality, we have
\begin{align}\label{fff}
T_1
\le & \left|\int_{\mathcal{O}} (u+f)(-\nabla\cdot \mathcal{R} - r_1 - r_2)\,\mathrm{d}x  \right| \nonumber\\
= & \left|\int_{\mathcal{O}} \mathcal{R}\nabla(u+f) - r_1(u+f) - r_2(u+f) \,\mathrm{d}x  \right| \nonumber\\
\le & \Vert \mathcal{R} \Vert^2_{L^{2}(\mathcal{O})} + \frac{1}{4} \Vert \nabla(u+f) \Vert^2_{L^{2}(\mathcal{O})} + \frac{1}{4(\Lambda-K^2-1/2)}\Vert r_1 \Vert^2_{L^{2}(\mathcal{O})} \nonumber\\
&+(\Lambda-K^2-1/2)\Vert u+f \Vert^2_{L^{2}(\mathcal{O})} + \left|\int_{\mathcal{O}} r_2(u+f) \,\mathrm{d}x \right|.
\end{align}
From (\ref{new Lam})-(\ref{fff}), we have
  \begin{align}\label{zui3 bdd}
&\frac{3}{4} \Vert \nabla u \Vert^2_{L^{2}(\mathcal{O})}+(\Lambda-K^2) \Vert u \Vert^2_{L^{2}(\mathcal{O})}\nonumber\\
\le & \Vert \mathcal{R} \Vert^2_{L^{2}(\mathcal{O})}+\frac{3}{8}\Vert \nabla u \Vert^2_{L^{2}(\mathcal{O})}+(\frac{1}{4}+4\lambda^2) \|\nabla f\|^2_{L^{2}(\mathcal{O})}+ \frac{1}{4(\Lambda-K^2-1/2)}\Vert r_1 \Vert^2_{L^{2}(\mathcal{O})}\nonumber\\
&+(\Lambda-K^2-1/4)\Vert u \Vert^2_{L^{2}(\mathcal{O})}+(\Lambda+3K^2+\Lambda^2-1/2)\Vert f \Vert^2_{L^{2}(\mathcal{O})}+\left |\int_{\mathcal{O}} r_2
  (u+f)\, \mathrm{d}x \right|.
\end{align}
As in the proof of Theorem \ref{R space}, we now address the last term on the right-hand side of (\ref{zui3 bdd}). Similarly, here we denote $(u+f)_\iota:=\fint_{z+\Box_\iota} (u+f)$. For the bounded domain \(\mathcal{O}\), the integral is decomposed as \(\int_{\mathcal{O}} \cdot \,\mathrm{d}x = \sum\limits_{z\in \iota \mathbb{Z}^d} \int_{z+\Box_\iota} \cdot \,\mathrm{d}x\), where the sum only includes those \(z\) for which \(z+\Box_\iota\) is contained within \(\mathcal{O}\). Similar to (\ref{pi}), by applying the Cauchy–Schwarz inequality and the Poincar\'e inequality, we have
\begin{align*}
    \left |\int_{\mathcal{O}} r_2
  (u+f)\, \mathrm{d}x \right|
  = &\left |\sum\limits_{z\in \iota \mathbb{Z}^d} \int_{z+\Box_\iota} \Gamma_i\partial_{i} u_0
  \left((u+f)-(u+f)_\iota\right) \,\mathrm{d}x + \sum\limits_{z\in \iota \mathbb{Z}^d} \int_{z+\Box_\iota} \Gamma_i\partial_{i} u_0
  (u+f)_\iota \,\mathrm{d}x \right|\nonumber\\
  \leq & \left| \sum\limits_{z\in \iota \mathbb{Z}^d} \left(\Vert \Gamma_i\partial_{i} u_0 \Vert_{L^{2}(z+\Box_\iota)} \iota \Vert \nabla(u+f) \Vert_{L^{2}(z+\Box_\iota)} + (u+f)_\iota \int_{z+\Box_\iota} \Gamma_i\partial_{i} u_0 \,\mathrm{d}x \right)\right|\nonumber\\
  \leq & \left| \sum\limits_{z\in \iota \mathbb{Z}^d} \Vert \Gamma_i\partial_{i} u_0 \Vert_{L^{2}(z+\Box_\iota)} \iota \Vert \nabla(u+f) \Vert_{L^{2}(z+\Box_\iota)} \right|+  \left| \sum\limits_{z\in \iota \mathbb{Z}^d} (u+f)_\iota \int_{z+\Box_\iota} \Gamma_i\partial_{i} u_0 \,\mathrm{d}x \right|\nonumber\\
  := & I_1+I_2.
\end{align*}
Similar to (\ref{pi term 1}), by applying Young's inequality, we have
\begin{align*}
     I_1
  \leq  \sum\limits_{z\in \iota \mathbb{Z}^d} \left(\Vert \Gamma_i\partial_{i} u_0 \Vert^2_{L^{2}(z+\Box_\iota)} \iota^2\right)+\frac{1}{4}
  \Vert \nabla(u+f) \Vert^2_{L^{2}(\mathcal{O})}.
\end{align*}
Similar to (\ref{pi term2}), by using the Cauchy–Schwarz inequality and Young's inequality, we have
\begin{align}\label{pi term2 bdd}
     I_2
  \leq  \frac{1}{4}\Vert u+f \Vert^2_{L^{2}(\mathcal{O})} + \sum\limits_{z\in \iota \mathbb{Z}^d} \iota^{-d} \left(\int_{z+\Box_\iota} \Gamma_i \partial_i u_0 \,\mathrm{d}x \right)^2.
\end{align}
From (\ref{zui3 bdd})-(\ref{pi term2 bdd}) and recalling that $u={u}^\epsilon- w^{\epsilon}$, we have
\begin{align}\label{r ll bdd}
     &\Vert \nabla({u}^\epsilon - w^{\epsilon}) \Vert^2_{L^{2}(\mathcal{O})} \nonumber\\
  \lesssim & \left(\Vert \mathcal{R} \Vert^2_{L^{2}(\mathcal{O})} + \Vert r_1 \Vert^2_{L^{2}(\mathcal{O})} + \sum\limits_{z\in \iota \mathbb{Z}^d} \Vert \Gamma_i \partial_i u_0 \Vert^2_{L^{2}(z+\Box_\iota)} \iota^2 + \sum\limits_{z\in \iota \mathbb{Z}^d} \iota^{-d} \left(\int_{z+\Box_\iota} \Gamma_i \partial_i u_0 \,\mathrm{d}x \right)^2 \right)\nonumber\\
 &+ \left(\|\nabla f \|^2_{L^{2}(\mathcal{O})} + \| f \|^2_{L^{2}(\mathcal{O})}\right) \nonumber\\
  := & E_1+E_2.
\end{align}
From Lemma \ref{pi test l2}-\ref{two-scale}, we have
\begin{align*}
E_1=& \Vert \mathcal{R} \Vert^2_{L^{2}(\mathcal{O})} + \Vert r_1 \Vert^2_{L^{2}(\mathcal{O})} + \Vert \Gamma_i \partial_i u_0 \Vert^2_{L^{2}(\mathcal{O})} \iota^2 + \sum\limits_{z\in \iota \mathbb{Z}^d} \iota^{d} \left(\fint_{z+\Box_\iota} \Gamma_i \partial_i u_0 \,\mathrm{d}x \right)^2 \nonumber\\
\lesssim & \mathcal{C}(\omega_\epsilon)\epsilon^2 \Vert \nabla u_0 \Vert_{H^1(\mathcal{O})}^2+\mathcal{C}(\omega_\epsilon) \Vert \nabla u_0 \Vert^2_{L^2(\mathcal{O})}\iota^2+\sum\limits_{z\in \iota \mathbb{Z}^d}\iota^{d}\mathcal{C}(\omega_\epsilon) \epsilon^d \iota^{-2d} \Vert \nabla u_0 \Vert_{L^2(z+\Box_\iota)}^2\nonumber\\
\lesssim & \mathcal{C}(\omega_\epsilon)\left(\epsilon^2 +\iota^2+ \epsilon^d \iota^{-d} \right)\Vert \nabla u_0 \Vert_{H^1(\mathcal{O})}^2.
\end{align*}
From (\ref{eta0}) and (\ref{eta}), and by applying Lemma \ref{phi sigma}, Lemma \ref{pi test l2}, and the trace-type estimate (\ref{trace type}), we obtain
\begin{align}\label{3.65}
      E_2\lesssim&\|(\nabla \eta_r)  \phi_{e_{i}} \partial_i {u}_0 \|^2_{L^{2}((\partial \mathcal{O})_\tau)}+ \|\nabla (   \phi_{e_{i}} \partial_i{u}_0 ) \|^2_{L^{2}((\partial \mathcal{O})_\tau)}+ \| \phi_{e_{i}} \partial_i {u}_0  \|^2_{L^{2}((\partial \mathcal{O})_\tau)}\nonumber\\
    \lesssim & \left(\frac{1}{\tau^2}+1\right)\| \phi_{e_{i}} \partial_i {u}_0  \|^2_{L^{2}((\partial \mathcal{O})_\tau)} + \|\phi_{e_{i}} \nabla (\partial_i {u}_0)   \|^2_{L^{2}((\partial \mathcal{O})_\tau)} + \|\partial_i {u}_0  \nabla \phi_{e_{i}} \|^2_{L^{2}((\partial \mathcal{O})_\tau)}
 \nonumber\\
    \lesssim &\left(\frac{1}{\tau^2}+1\right)\mathcal{C}(\omega_\epsilon)\epsilon^2 \Vert \nabla u_0 \Vert^2_{L^{2}((\partial \mathcal{O})_\tau)} + \mathcal{C}(\omega_\epsilon)\epsilon^2 \Vert \nabla u_0 \Vert_{H^1(\mathcal{O})} ^2+ \mathcal{C} (\omega_\epsilon) \Vert \nabla u_0 \Vert^2_{L^{2}((\partial \mathcal{O})_\tau)} \nonumber\\
    \lesssim & \mathcal{C}(\omega_\epsilon) \left(\frac{\epsilon^2}{\tau}+\tau\epsilon^2+\epsilon^2+\tau\right) \Vert \nabla u_0 \Vert_{H^1(\mathcal{O})} ^2.
\end{align}
Now, from (\ref{r ll bdd})-(\ref{3.65}), we obtain
\begin{align}\label{r lll2}
     \Vert \nabla({u}^\epsilon- w^{\epsilon}) \Vert^2_{L^{2}(\mathcal{O})} 
   \lesssim  \mathcal{C}(\omega_\epsilon)\left(\epsilon^2 +\iota^2+ \epsilon^d \iota^{-d}+\frac{\epsilon^2}{\tau}+\tau\epsilon^2+\tau \right)\Vert \nabla u_0 \Vert_{H^1(\mathcal{O})}^2.
\end{align}

By combining (\ref{r lll2}) and (\ref{r left left}) with Lemma \ref{phi sigma}, we obtain
\begin{align}\label{last2}
		& \Vert u^{\epsilon}-u_0\Vert_{L^{2}(\mathcal{O})}\nonumber\\
  \lesssim & \mathcal{C}(\omega_\epsilon) \left(\epsilon +\iota+ \epsilon^{d/2} \iota^{-d/2}+\epsilon\tau^{-1/2}+\tau^{1/2}\epsilon+\tau^{1/2} \right)\Vert \nabla u_0 \Vert_{H^1(\mathcal{O})}+\Vert  \phi_{e_{i}} \partial_{i} {u}_0\Vert_{L^{2}(\mathcal{O})}\nonumber\\
  \lesssim & \mathcal{C}(\omega_\epsilon) \left(\epsilon +\iota+ \epsilon^{d/2} \iota^{-d/2}+\epsilon\tau^{-1/2}+\tau^{1/2}\epsilon+\tau^{1/2} \right)\Vert \nabla u_0 \Vert_{H^1(\mathcal{O})}.
	\end{align} 
We follow the same approach to minimize the expression \(\epsilon +\iota+ \epsilon^{d/2} \iota^{-d/2}+\epsilon\tau^{-1/2}+\tau^{1/2}\epsilon+\tau^{1/2}\) as in the whole space. We assume \(\iota = \epsilon^\alpha\) and choose \(\alpha\) such that \(\iota\) and \(\epsilon^{d/2} \iota^{-d/2}\) are balanced. At the same time, we choose \(\tau = \epsilon^\beta\) and select \(\beta\) such that \(\epsilon\tau^{-1/2}\) and \(\tau^{1/2}\) are balanced. By taking \(\iota = \epsilon^{\frac{d}{d+2}}\) and \(\tau = \epsilon\), and noting $d\ge 3$, we obtain from (\ref{last2}) that
\begin{align*}
		\Vert u^{\epsilon}-u_0\Vert_{L^{2}(\mathcal{O})}\lesssim  \mathcal{C}(\omega_\epsilon) \epsilon^{\frac{1}{2}}\Vert \nabla u_0 \Vert_{H^1(\mathcal{O})}.
	\end{align*}  
 This completes the proof of (\ref{bounded result}).
\end{proof}

%% file: gragh.tex
\section{Examples}\label{sec 5}
Before presenting our own example to illustrate the spectral gap condition, we refer to a similar example discussed in \cite{Fischer2021}. In that work, the authors consider a stationary Gaussian random field with a rapidly decaying covariance function that satisfies the spectral gap condition. Specifically, they demonstrate that for any stationary Gaussian random field with an integrable covariance function, the spectral gap condition can be satisfied, ensuring controlled local dependencies. Building upon this, we consider a Gaussian random field \(\omega_\epsilon(x):\) \(\mathbb{R}^d\to H\) with the following covariance function:
\begin{equation}\label{short}
\mathrm{Cov}(\omega_\epsilon(x), \omega_\epsilon(y))  = \exp\left(-\frac{|x - y|^2}{2\epsilon^2}\right).   
\end{equation}
Here, \(\omega_\epsilon(x)\) is a stationary Gaussian random field with a correlation length \(\epsilon \in (0, 1]\). The covariance function decays rapidly as the distance \(|x - y|\) increases, which ensures that the random field satisfies the spectral gap condition. Specifically, this rapid decay in the covariance function implies that the correlations between the values of the random field at different points diminish quickly. This rapid decorrelation is essential for the spectral gap condition. According to \cite[Theorem 3.1]{Duerinckx}, such a Gaussian random field satisfies the spectral gap condition. The theorem states that for a stationary Gaussian field with an integrable covariance function, the variance of any measurable random variable can be controlled by its local dependence on \(\omega_\epsilon\).

To further illustrate the spectral gap condition in the context of Gaussian random fields, we compare two types of covariance structures: short-range and long-range correlations. The left side of the Figure \ref{fig:covariance_comparison} shows the previously mentioned Gaussian random field \(\omega_\epsilon(x)\) with the covariance function (\ref{short}), which exemplifies short-range correlations. On the right side, we present a Gaussian random field with long-range correlations, characterized by a different covariance function:
\[
\mathrm{Cov}(\omega_\epsilon(x), \omega_\epsilon(y)) = \frac{1}{\left(1 + \frac{|x - y|}{\epsilon}\right)^{1/2}}.
\]
In this case, the correlation between \(\omega_\epsilon(x)\) and \(\omega_\epsilon(y)\) decays more slowly, indicating stronger dependencies across larger distances.
\begin{figure}[H]
\centering
\includegraphics[width=0.42\textwidth]{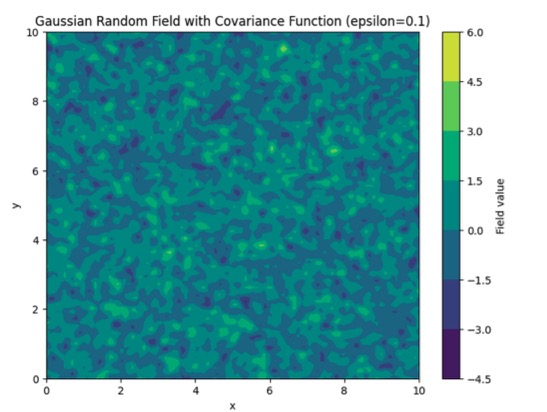}
\includegraphics[width=0.38\textwidth]{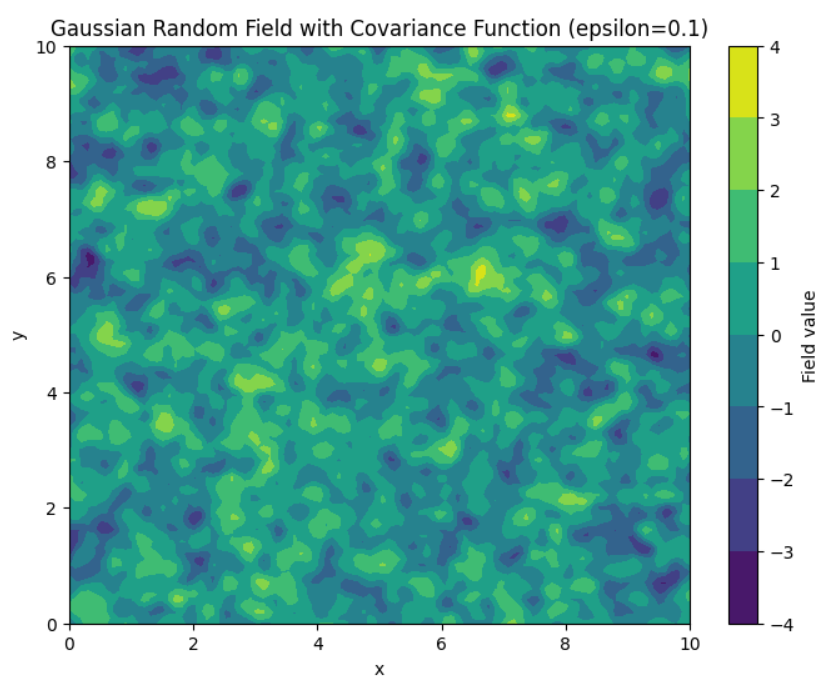}
\caption{Comparison of Gaussian random fields with short-range (left) and long-range (right) correlations. The covariance functions are indicated above each field.}
\label{fig:covariance_comparison}
\end{figure}

To gain further insights into how the parameter \(\epsilon\) influences the correlation structure of Gaussian random fields, we visualize two instances with different values of \(\epsilon\) in Figure \ref{fig:comparison}. This comparison highlights the differences in correlation lengths and their implications for the spectral gap condition. Figure \ref{fig:comparison} illustrates the effect of different correlation lengths on Gaussian random fields. The left panel corresponds to \(\epsilon = 0.1\), where the field exhibits short-range correlations, leading to rapid changes in field values over small distances. In contrast, the right panel corresponds to \(\epsilon = 1.0\), which induces longer-range correlations, resulting in smoother variations across the field. 

\begin{figure}[H]
\centering
\includegraphics[width=0.42\textwidth]{0.1.jpg}
\includegraphics[width=0.4\textwidth]{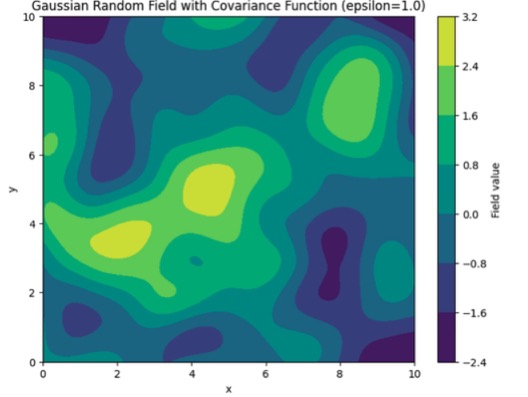}
\caption{Comparison of Gaussian Random Fields with Different Correlation Lengths.}
 \label{fig:comparison}
\end{figure}